\documentclass[a4paper,11pt,reqno]{amsart}
\usepackage{upref}
\usepackage[mathcal]{euscript}

\usepackage{amsfonts,amssymb}

\usepackage{amsthm}

\usepackage{amsmath}

\usepackage{amscd}
\usepackage{tabularx}
\usepackage{arydshln}
\usepackage{verbatim}

 \usepackage{a4wide}

\theoremstyle{plain}

\newtheorem{theorem}{Theorem}[section]
\newtheorem{lemma}[theorem]{Lemma}
\newtheorem{proposition}[theorem]{Proposition}
\newtheorem{corollary}[theorem]{Corollary}
\newtheorem{Counter-example}[theorem]{Counter-example}
\newtheorem{remark}[theorem]{Remark}
\newtheorem{example}[theorem]{Example}

\theoremstyle{definition}

\newtheorem{definition}[theorem]{Definition}

\theoremstyle{remark}

\usepackage{amssymb}

\def\R{\mathbb R}

\def\C{\mathbb C}
\def\SS{\mathbb S}
\def\H{\mathbb H}

\newcommand{\GL}{\mathrm{GL}}

\newcommand{\SU}{\mathrm{SU}}

\newcommand{\SO}{\mathrm{SO}}

\newcommand{\SP}{\mathrm{Sp}}
\newcommand{\Ad}{\mathrm{Ad}}

\newcommand{\id}{\mathrm{id}}
\def\hom{\mathrm{Hom}}

\def\dim{\mathop{\hbox{\rm dim}}}

\begin{document}

\title[Invariant affine connections on odd dimensional   spheres]{Invariant affine connections \\on odd dimensional   spheres}

\author[C.~Draper]{Cristina Draper${}^\star$}
\address{Departamento de Matem\'atica Aplicada, 
%Escuela de las Ingenier\'{\i}as, 
Universidad de M\'alaga, 
%Ampliaci\'on Campus de Teatinos, 29071 M\'alaga, 
Spain}
%\curraddr{}
\email{cdf@uma.es}
\thanks{${}^\star$ Supported by the Spanish MEC grant MTM2013-41208-P and by the Junta de Andaluc\'{\i}a grants FQM-336, FQM-7156, with FEDER funds.}  %La primera financiada también por MINECO, no sé si hay que ponerlo

\author[A.~Garv\'{\i}n]{Antonio Garv\'{\i}n${}^\ast$}
%\curraddr{}
\email{garvin@uma.es}
\thanks{${}^\ast$ Supported by the
Spanish MEC grant MTM2013-41768-P and by the Junta de Andaluc\'{\i}a grant  FQM-213, with FEDER funds.
}

\author[F.J.~Palomo]{Francisco J. Palomo${}^\dagger$}
%\curraddr{}
\email{fjpalomo@ctima.uma.es}
\thanks{${}^\dagger$ Supported by   the Spanish MEC grant MTM2013-47828-C2-2-P
and by the Junta de Andaluc\'{\i}a  grant  FQM-4496, with FEDER funds.}



\date{}

%%%%%%%%%%%%%%%%%%%%%%%%%%%
%%%%%%%%%%%%%%%%%%%%%%%%%%%

\begin{abstract}
A Riemann-Cartan manifold is a  Riemannian manifold endowed with an affine connection which is compatible with the metric tensor.
This affine connection  is not necessarily torsion free. Under the assumption that the manifold is a homogeneous space, the notion of homogeneous Riemann-Cartan space is introduced in a natural way.
For the case of the odd dimensional spheres $\SS^{2n+1}$ viewed as homogeneous spaces of the special unitary groups,  the classical Nomizu's Theorem on invariant connections has permitted to obtain
an algebraical description of all the connections which turn the spheres $\SS^{2n+1}$ into   homogeneous Riemann-Cartan spaces. 
The expressions  of such connections as covariant derivatives are given by means of several invariant tensors: 
the ones of the usual Sasakian structure of the sphere;
   an invariant 3-differential form coming from  a $3$-Sasakian structure on $\SS^7$;
 and  the involved ones in the almost contact metric structure of  $\SS^5$ 
 provided by  its natural embedding into the nearly K\"ahler  manifold $\SS^6$.
 Furthermore, the invariant connections sharing geodesics with the Levi-Civita one have also been completely described.  Finally,  $\SS^3$ and $\SS^7$ are characterized as the unique odd-dimensional spheres which admit nontrivial invariant connections satisfying an Einstein-type equation.
\end{abstract}

\maketitle

%%%%%%%%%%%%%%%%%%%%%%%%%%%%%%
%%%%%%%%%%%%%%%%%%%%%%%%%%%%%%

\section{Introduction}

This work is mainly devoted to give a method to construct Riemann-Cartan manifolds from the theory of invariant connections on homogeneous spaces, as well as to use algebraical tools   to obtain remarkable properties of the involved affine connections. 
In order to explain our point of view, let us first recall some general facts about both topics.

On the  one hand, the space of invariant connections of a homogeneous space $M=G/H$ has been well understood since the classical paper \cite{Draper:teoNomizu}. Indeed, starting from a fixed reductive decomposition $\mathfrak{g}=\mathfrak{h}\oplus \mathfrak{m}$ of the Lie algebra $\mathfrak{g}$ of $G$, Katsumi Nomizu established
a very fruitful one-to-one correspondence between the set of all invariant connections on $M$ and the set of all bilinear functions $\alpha$ on $\mathfrak{m} $ with values in $\mathfrak{m}$ which are invariant by $\mathrm{Ad}(H)$ (Section~2). From an algebraic point of view, the search of invariant connections on $M$ reduces to find the algebra structures $\alpha$ on $\mathfrak{m}$ which satisfy the condition $\mathrm{Ad}(H)\subset\mathrm{Aut}(\mathfrak{m},\alpha)$. Thus, we have to look for algebra structures $\alpha$ on $\mathfrak{m}$ which have a fixed subgroup of automorphisms.  In the terminology of \cite{Draper:S6}, $(\mathfrak{m},\alpha)$ is called the
\emph{connection algebra} of the corresponding invariant connection. In general, the properties of the connection algebras are not very well-known, except for a few cases. For instance, under the assumption that $M$ is a symmetric space and $G$ is simple,   then $\mathfrak{m}$ can be identified with the traceless elements of a simple 
finite-dimensional Jordan algebra and, up to scalars,  $\alpha$ corresponds with the projection of the Jordan product  \cite{Draper:losLts}. Also, for $M=\SS^6$ treated as $G_2$-homogeneous space, the related connection algebra  $(\mathfrak{m},\alpha)$ turns out to be a color algebra (see details in \cite{Draper:S6}).

On the other hand,   a Riemann-Cartan manifold is a triple $(M,g,\nabla)$ where $(M,g)$ is a Riemannian manifold and $\nabla$ is a metric affine connection (i.e., $\nabla g=0$). Thus, a nonzero torsion tensor $T^\nabla$ is allowed. 
Recall that a slight variation on the classical argument based on the Koszul formula actually shows that every metric affine connection is determined by its torsion tensor.
\'{E}.~Cartan was the first one to investigate affine connections with nonvanishing torsion tensor. From a mathematical approach, he appealed  to consider not only the Levi-Civita connection on a Riemannian manifold. In fact, he demanded in 1924 \cite{Cartancita}: {\it  \lq \lq given a manifold ... attribute to this manifold the affine connection that reflects in the simplest way the relations  of this manifold with the ambient space\rq \rq}. Roughly speaking, the connection should be adapted to the geometry under consideration. 
For illustrating this fact, let us recall at this point the characteristic connection  $\nabla^c$ on every Sasaki manifold $M$ (Section~4). This connection is metric and preserves the tensor fields involved in the definition of a Sasaki manifold   (Example~\ref{222}). The   connection  $\nabla^c$ is not free-torsion in general, so that the manifold $M$ is moreover endowed with the corresponding torsion tensor.
For a general Riemann-Cartan manifold $(M,g,\nabla)$, the type of the torsion tensor can be classified algebraically (see, for instance, \cite{Draper:Tricerri}). The main type for us will be the case of skew-symmetric torsion. Under this assumption, the connection $\nabla$ has the same geodesics as the Levi-Civita connection $\nabla^{g}$ and the manifold $M$ is endowed with a $3$-differential form $\omega_{_\nabla}$ (\ref{tos}).  
The excellent survey \cite{Draper:surveyagricola} provides a bridge to realize how \lq\lq{\it the skew-symmetric torsion enters in scene as one of the main tools to understand nonintegrable geometries}\rq\rq. This survey includes both mathematical and physical motivations as well as a wide variety of examples of metric connections 
 with torsion.  Another  approach, mainly devoted to the study of geometric vector fields and integral formulas on Riemann-Cartan manifolds, can be found in \cite{Draper:losrusosRiemCar}.

The existence of a metric affine connection which satisfies certain additional conditions permits to characterize remarkable geometric properties. In fact,   Ambrose-Singer Theorem states that a connected, complete and simply connected Riemannian manifold $(M,g)$ is homogeneous
if and only if there is a metric affine connection $\nabla$ such that $\nabla R^g=0$ and $\nabla T^{\nabla}=0$, where $R^g$ denotes the curvature tensor of the Levi-Civita connection and $T^{\nabla}$ the torsion tensor of $\nabla$ (see  \cite{Draper:Tricerri} and references therein).

Focussing now on our aim, let us assume $M=G/H$ is endowed with a homogeneous Riemannian metric $g$. Thus, the Lie group $G$ acts transitively by $g$-isometries on the manifold $M$. For an arbitrary affine connection $\nabla$ on $M$, there are two natural conditions to be imposed.  On the one hand, the homogeneity property seems to require that $\nabla$ is {\it preserved} by the action of $G$ 
 and on the other one, it seems natural to claim that the parallel transport associated to $\nabla$ is a $g$-isometry (i.e., $\nabla g=0$). These two properties lead to the notion of homogeneous Riemann-Cartan manifold (Definition \ref{HRCM}). Our main goal 
of study in this paper  is the canonical odd dimensional sphere $\SS^{2n+1}$ viewed as homogeneous space of the special unitary group $\SU(n+1)$, that is,
$\SS^{2n+1}=\SU( n+1)/\SU(n)$ ($n\ge1$). At this point it is interesting to remark that there is only one invariant (metric) connection for $\SS^{n}=\SO(n+1)/\SO(n)$: the Levi-Civita connection. Thus, the odd dimensional spheres as coset of special unitary groups can be seen as the easier nontrivial example of homogeneous Riemann-Cartan manifold on spheres. Indeed, there are $\SU(n+1)$-invariant metric connections on $\SS^{2n+1}$ which are different from the Levi-Civita connection.
We obtain in theorems \ref{principal4},  \ref{principal7},  \ref{principalS5} and \ref{principalS3},
the explicit expressions for such $\SU(n+1)$-invariant metric connections on
$\SS^{2n+1}$ ($n\ge4$, $n=3$, $n=2$ and $n=1$, respectively).  
As far as we know, it is not usual in the literature to recover as covariant derivatives the invariant connections from the algebraic data (i.e., bilinear operations $\alpha$ on $\mathfrak{m}$).

As   was mentioned, a specially interesting case of metric connections corresponds to those with skew-symmetric torsion.
Thus, the following two questions arise in a natural way.
\begin{quote}
{\it 
How many $\SU(n+1)$-invariant metric connections on $\SS^{2n+1}$ with skew-symmetric torsion are there?

\noindent Which are the tensor  fields which permit to give explicit expressions of such connections?}
\end{quote}

The answer to these questions depends on the dimension of the sphere $\SS^{2n+1}$ as follows. The space of $\SU(n+1)$-invariant metric  connections with skew-symmetric torsion has one free parameter except for $\SS^5$ and $\SS^7$ and, in both cases, such space depends on three free parameters. For the spheres $\SS^{2n+1}$ with $n\neq 2, 3$,  the $3$-differential form $\omega_{_{\nabla}}$ obtained from the torsion of the invariant metric connection $\nabla$ is proportional to $\eta \wedge d\eta$ (Remark \ref{111}). Here $\eta$ is the $1$-differential form metrically equivalent to the Hopf vector field $\xi$ (\ref{elcampo}). Thus, these connections can  explicitly be written from the usual Sasakian structure on $\SS^{2n+1}$ (Sections 3 and 4).

For the $\SS^5$-case, in order to find the explicit expressions of the invariant connections,
   it is necessary to use, besides   the Sasakian structure, an almost contact metric structure different from the usual one. This almost contact structure (but not contact) is closely related with the nearly K\"{a}hler structure on $\SS^6$. Thus, the great amount of invariant connections on $\SS^5$ is caused by the almost contact metric structure obtained when seeing  $\SS^5$ as a totally geodesic hypersurface of  $\SS^6$  (Section~6).
 
We think that the most interesting connections arise for the sphere $\SS^7$. In this case, the main difficulty to achieve the explicit expressions for the metric invariant connections lies in finding a $\SU(4)$-invariant $3$-differential form on $\SS^7$ which permits to write down the   connections. We have considered the canonical $3$-Sasakian structure on $\SS^7$ and we have introduced the   $3$-differential form 
on $\SS^7$ given by
$$
\Omega=\frac{1}{2}( \eta_{2}\wedge d\eta_{2}- \eta_{3}\wedge d\eta_{3}),
$$
where $\eta_{2}$ and $\eta_{3}$ are the $1$-differential forms associated with  the other Sasakian structures on $\SS^7$ (Section 5). The most technical part of this paper is devoted to show, in lemmas~\ref{algebra2}, \ref{le_positivo}, \ref{algebra} and Proposition~\ref{invarianzaCris}, that the $3$-differential form $\Omega$ is $\SU(4)$-invariant.
The $3$-differential forms $\omega_{_{\nabla}}$ obtained from the torsions of all the invariant metric connections with skew-torsion $\nabla$ on $\SS^7$  are now given by
\begin{equation}\label{eq_lastorsionesS7}
\omega_{_{\nabla}}=\frac12 r  \, \eta_{1}\wedge d\eta_{1}+ \mathrm{Re}(q) \left( \eta_{2}\wedge d\eta_{2}- \eta_{3}\wedge d\eta_{3}\right)- \mathrm{Im}(q) \left(\eta_{2}\wedge d\eta_{3}+\eta_{3}\wedge d\eta_{2}\right),
\end{equation}
for $r\in \R$ and $q\in \C$. Hence (\ref{eq_lastorsionesS7}) parametrizes the vector space of the $\SU(4)$-invariant 3-differential forms on $\SS^7$. Recall that a different $3$-differential form, called the \emph{canonical} $G_{2}$-\emph{structure}, has been considered for arbitrary $7$-dimensional $3$-Sasakian manifolds  \cite{AgriFri} and there is a unique metric (characteristic) connection preserving the $G_{2}$-structure with   skew-symmetric torsion. The family of $\SU(4)$-invariant connections on $\SS^7$ does not contain the characteristic connection of the  canonical $G_{2}$-structure on $\SS^7$ (Remark~\ref{remark_conexG2}).
From the $3$-differential form $\Omega$, we have introduced a skew-symmetric bilinear map $\Theta$ on $\xi^{\perp}$ which satisfies $g(\Theta (X,Y), X)=0$ for all $X,Y\in \xi^{\perp}$ (Remark \ref{operaciontheta}).
\smallskip

Paraphrasing R.~Thom: \emph{Are there any best (or nicest, or distinguished) Riemann-Cartan structure on a manifold $M$?} (see \cite[Introduction]{Besse}). An answer has been proposed in \cite{AgriFerr} with the notion of \lq\lq Einstein manifold with skew-torsion\rq\rq  as follows.
For an arbitrary Riemann-Cartan manifold $(M,g,\nabla)$, the usual notions of Ricci tensor $\mathrm{Ric}^{\nabla}$ and scalar curvature $s^{\nabla}$ have natural  generalizations. 
Then, a Riemann-Cartan manifold $(M,g,\nabla)$ with skew-symmetric torsion is said to be  \emph{Einstein with skew-torsion} whenever
\begin{equation}\label{einstein}
\mathrm{Sym}(\mathrm{Ric}^{\nabla})=\frac{s^{\nabla}}{\dim M}\, g,
\end{equation}
where $\mathrm{Sym}$ denotes the symmetric part of the corresponding tensor. This notion is also deduced from a variational principle and, as   is expected, it reduces to the usual notion of Einstein manifold when one considers the Levi-Civita connection. The sphere $\SS^{2n+1}$ with its canonical Riemannian metric has constant sectional curvature $1$ and hence it is trivially  an Einstein manifold in the usual sense.
Thus, we arrive to the following question.
\begin{quote}
{\it Are there nontrivial examples of $\SU(n+1)$-invariant metric connections on $\SS^{2n+1}$ such that $(\SS^{2n+1},g,\nabla)$ is Einstein with skew-torsion?
}
\end{quote}
Again the answer strongly depends on the dimension of the sphere $\SS^{2n+1}$. In order to check the $\nabla$-Einstein equation (\ref{einstein}), we have computed the symmetric part of the Ricci tensors of all $\SU(n+1)$-invariant metric connections with skew-symmetric torsion. It turns out that, for $n\geq 4$ and $n=2$, the only $\SU(n+1)$-invariant metric connection with nonvanishing skew-symmetric torsion which satisfies the $\nabla$-Einstein equation is the Levi-Civita connection.

For dealing with the spheres $\SS^7$ and $\SS^3$, first recall that a classical result by \'{E}.~Cartan and J.A.~Schouten states that a Riemannian manifold $M$ which admits a flat metric connection with totally skew-torsion splits and each irreducible factor is either a compact simple Lie group or the sphere $\SS^7$ \cite{CartanSchouten} (see also \cite{AgriFri2}). Taking into account that $\SS^7$ and $\SS^3$ are parallelizable manifolds by Killing vector fields, it is  possible to give examples of flat metric connections with skew-torsion on $\SS^7$ and $\SS^3$ (details in Remark~\ref{flatEinstein}). These flat metric connections trivially satisfy the $\nabla$-Einstein equation. In this paper, we have found stricking new families of invariant metric connections   on $\SS^7$ and $\SS^{3}$  satisfying the $\nabla$-Einstein equation which are not all of them flat. 
Indeed, for  $\SS^7$ and for each choice of parameters $r\in \R$ and $q\in \C$ with $|q|^2=r^2$, we have obtained in Corollary~\ref{777} one $\SU(4)$-invariant metric connection with skew-symmetric torsion which satisfies the $\nabla$-Einstein equation. In the particular case $r=0$, we recover the Levi-Civita connection $\nabla^{g}$. Several comments about flatness are compiled in Table~\ref{tablaRiccis}. 

 For the  $\SS^3$-case, there is a one free parameter family of invariant metric connections with skew-symmetric torsion   satisfying the $\nabla$-Einstein equation. It is interesting to point out that, formally, this family is the same one as obtained  with skew-symmetric torsion on every odd dimensional sphere $\SS^{2n+1}$ for $n\geq 4$, nevertheless  the $\nabla$-Einstein equation is not satisfied unless $n=1$.

\smallskip

This paper is organized as follows. First, Sections 2 and 3 give the basic background. Several well-known facts are shortly recalled in order to fix some notation since we have tried to keep the paper as self-contained as possible. Section~2 is mainly devoted to homogeneous spaces and to the basic definitions on Riemann-Cartan manifolds. Then, for the sake of completeness,  Nomizu\rq{s} Theorem is stated as   will be used later. Section~2 ends with the characterization of the bilinear operations $\alpha\colon \mathfrak{m}\times \mathfrak{m}\to \mathfrak{m}$ which correspond  to  invariant metric  connections. Although  this characterization can be concluded from a broader setting, we include here an {\it ad hoc} proof. Section~3 contains   the basic geometrical facts on the odd dimensional spheres $\SS^{2n+1}$ treated as homogeneous spaces $\SU(n+1)/\SU(n)$.  Particularly, we focus on the usual Sasakian structure of $\SS^{2n+1}$ and on the Hopf map. The Sasakian structure of $\SS^{2n+1}$ is the key tool to describe as covariant derivatives  the $\SU(n+1)$-invariant connections for $n\geq 4$ and $n=1$. This section also encloses usual conventions to be used throughout the article.
Section~4 is devoted simultaneously to all the odd dimensional spheres such that $n\geq 4$. The space of invariant connections has seven free parameters which reduce to three when we consider metric connections and to one for the case of skew-symmetric torsion.   These connections are explicitly described in Theorem~\ref{principal4}.
The remaining sections  deal with   low dimensional spheres and they are structured in the same way as  Section~4. The  sphere $\SS^7$ is studied in Section~5, which includes a remarkable   result   stating the invariance of the $3$-differential form $\Omega$ (Proposition~\ref{invarianzaCris}). This result requires several additional lemmas of linear algebra. Sections 5 and 6 are concerned with $\SS^5$ and $\SS^3$ respectively.
The work finishes with an appendix where   the main results   are compiled in a table. Some of these results   were partially announced in \cite{nuestroPadge}.

%%%%%%%%%%%%%%%%%%%%%%%%%%%%%%
%%%%%%%%%%%%%%%%%%%%%%%%%%%%%%\Black

\section{Preliminaries}

Let $M$ be a (smooth) manifold and $G$ a Lie group which acts transitively on the left on $M$.  As usual, we write a dot to denote the action of $G$ on $M$ and so, for $\sigma \in G$, the left translation by $\sigma$ will be given by $\tau_{\sigma}(p)=\sigma \cdot p$ for all $p\in M$. Fix a point $o \in M$ and consider the isotropy subgroup $H$ at $o$. It is well-known that the map
$ G/H\to M$ given by $\sigma\,H \mapsto \sigma\cdot o$ is a diffeomorphism, where $G/H$ is the set of left cosets modulo $H$ considered with the unique manifold structure such that the natural projection $\pi\colon G\to G/H$ is a submersion. The manifold $M$ is called a $G$-homogeneous space. For each $\sigma \in G$ and $X\in \mathfrak{X}(M)$, the vector field $\tau_{\sigma}(X)\in \mathfrak{X}(M)$ is given at every $p\in M$ by
$$
 (\tau_{\sigma}(X) )_p:=(\tau_{\sigma})_{*}(X_{\sigma^{-1}\cdot p}).
$$
That is, $\tau_{\sigma}(X)$ is the unique vector field on $M$ such that the following diagram commutes
\begin{equation*} 
\begin{CD}
T M @> (\tau_{\sigma})_{*}>> T M\\
@A X   AA @AA\tau_{\sigma}(X)A \\
 M @>\tau_{\sigma}>>  M.
\end{CD}
\end{equation*}
Let $\mathfrak{g}$  be the Lie algebra of the left invariant vector fields on $G$. For every $A\in \mathfrak{g}$, we denote by $A^{+}\in \mathfrak{X}(M)$ the vector field given at $p\in M$ by
$$
A^{+}_p:=\frac{d(\mathrm{exp}\,tA \cdot p)}{dt}\mid_{t=0}.
$$
The map $A\mapsto A^{+}$ provides an antihomomorphism of Lie algebras from $\mathfrak{g}$ to $\mathfrak{X}(M)$. The Lie subalgebra $\mathfrak{g}^{+}=\{A^+\mid A\in \mathfrak{g}\}\le \mathfrak{X}(M)$ locally spans all $\mathfrak{X}(M)$. That is, every point $p\in M$ has an open neighborhood $V$ such that for every $X\in \mathfrak{X}(M)$ there are smooth functions $f_{j}$ on $V$ and $A_{j}\in \mathfrak{g}$ with $X\vert _{V}=\sum f_{j}{A^{+}_{j}}\vert_{V}$.
The following formula holds for each $\sigma \in G$  and $p\in M$,
\begin{equation}\label{AD}
(\tau_{\sigma})_{*}(A^{+}_p)=(\mathrm{Ad}_{\sigma}A)^{+}_{\sigma\cdot p},
\end{equation}
where $\mathrm{Ad}$ denotes the adjoint representation of the Lie group $G$.
An affine connection $\nabla$ on $M$ is said to be $G$-\emph{invariant}  if, for every $\sigma \in G$ and for all $X,Y\in\mathfrak{X}(M)$,
$$
\tau_{\sigma}(\nabla_{X}Y)=\nabla_{_{\tau_{\sigma}(X)}}\tau_{\sigma}(Y).
$$

\vspace{0.2cm}

Recall \cite{Draper:surveyagricola} that a \emph{Riemann-Cartan manifold} 
is a triple $(M,g,\nabla)$ where $(M,g)$ is a Riemannian manifold and $\nabla$ is a \emph{metric} affine connection. That is,
$
X(g(Y,Z))=g(\nabla_{X}Y,Z)+g(Y,\nabla_{X}Z)
$
holds for all $X,Y,Z\in \mathfrak{X}(M)$. The \emph{torsion} tensor field of $\nabla$ is defined by $T^{\nabla}(X,Y)=\nabla_{X}Y-\nabla_{Y}X-[X,Y]$ for $X,Y\in \mathfrak{X}(M)$ and it is not assumed to be $T^{\nabla}=0$, in general. Therefore, Riemann-Cartan manifolds can be seen as a generalization of the usual Riemannian manifolds where
the metric affine connection under consideration is always  the Levi-Civita connection $\nabla^{g}$, which is characterized by the condition $T^{\nabla^{g}}=0$.
Contractions of the Riemann curvature\footnote{Our convention on the sign is
$R(X,Y)Z=\nabla_{X}
\nabla_{Y}Z-\nabla_{Y}\nabla_{X}Z-\nabla_{[X,Y]}Z$.} of $\nabla$ yield the usual invariants: the Ricci curvature tensor $\mathrm{Ric}^{\nabla}$ and the scalar curvature $s^{\nabla}$. As usual, quantities referring to the Levi-Civita connection $\nabla^g$ will carry an index $g$ and quantities associated with a metric affine connection $\nabla$ will have an   index $\nabla$. 

The \emph{difference tensor} between the Levi-Civita connection $\nabla^g$ and an arbitrary affine connection (metric or not) $\nabla$ is the $(1,2)$-tensor field $\mathcal{D}$ given by
$
\nabla_{X}Y=\nabla^{g}_{X}Y+\mathcal{D}(X,Y)
$
for $X,Y\in \mathfrak{X}(M)$.
The torsion tensor $T^{\nabla}$ of the affine connection $\nabla$ satisfies $T^{\nabla}(X,Y)=\mathcal{D}(X,Y)-\mathcal{D}(Y,X)$. Observe that  $\nabla$ has the same geodesics as the Levi-Civita connection $\nabla^{g}$ if and only if $\mathcal{D}$ is skew-symmetric. In such a case, $T^{\nabla}(X,Y)=2(\nabla_{X}Y-\nabla^{g}_{X}Y)$ holds.  

Assume now $(M,g,\nabla)$ is a Riemann-Cartan manifold and set
\begin{equation}\label{tos}
\omega_{_\nabla}(X,Y,Z):=g(T^{\nabla}(X,Y),Z),
\end{equation}
for $X,Y,Z\in \mathfrak{X}(M)$. Then, the connection $\nabla$ is said
to have \emph{totally skew-symmetric torsion} (briefly, skew-torsion) if $\omega_{_\nabla}$ defines a $3$-differential form on $M$. It is an easy matter to show that a metric affine connection $\nabla$
shares geodesics with the Levi-Civita connection $\nabla^{g}$ if and only if $\nabla$ has totally skew-symmetric torsion.  

\vspace{0.2cm}

Under the additional assumption that $M$ is a $G$-homogeneous space, the following notion arises in a natural way.

\begin{definition}\label{HRCM}
 $(M,g,\nabla)$ is called a $G$-\emph{homogeneous Riemann-Cartan  space}  when $M$ is a $G$-homogeneous space endowed with a $G$-invariant Riemannian metric $g$ (i.e., $\tau_{\sigma}$ is a $g$-isometry for all $\sigma\in G$) and a $G$-invariant metric affine connection $\nabla$.
\end{definition}

\vspace{0.2cm}
Now we state   Nomizu's Theorem \cite{Draper:teoNomizu} on $G$-invariant affine connections in a suitable way for our aims. Several definitions and notations are required. A homogeneous space $M=G/H$ is said to be \emph{reductive} if the Lie algebra $\mathfrak{g}$ of $G$ admits a vector space decomposition
\begin{equation}\label{eq_descomposicionreductiva}
\mathfrak{g}=\mathfrak{h}\oplus\mathfrak{m},
\end{equation}
for $\mathfrak{h}$ the Lie algebra of $H$ and $\mathfrak{m}$ an $\text{Ad}(H)$-invariant subspace (i.e., $\text{Ad}(H)(\mathfrak{m})\subset \mathfrak{m}$). In this case, $\mathfrak{g}=\mathfrak{h}\oplus\mathfrak{m}$ is called a reductive decomposition of $\mathfrak{g}$. The condition $\text{Ad}(H)(\mathfrak{m})\subset \mathfrak{m}$ implies that $[\mathfrak{h},\mathfrak{m}]\subset\mathfrak{m}$, and both are equivalent when $H$ is connected. The differential map $\pi_{*}$ of the projection $\pi\colon G\to M=G/H$ gives a linear isomorphism
$(\pi_{*})_e\vert_{\mathfrak{m}}\colon\mathfrak{m}\to T_{o}M$. Note that $\pi_{*}(A)=A^{+}_{o}$ for all $A\in \mathfrak{m}$. The isotropy representation $H \to \GL(T_{o}M)$ given by $\sigma \mapsto (\tau_{\sigma})_{*}$ corresponds under $\pi_{*}$ to $\text{Ad}\colon H\to  {\GL}(\mathfrak{m})$, according to Equation~(\ref{AD}).

\begin{remark}
\rm {It is often said that a space $M$ is \lq\lq reductive\rq\rq, but this is an abuse of notation. The reductivity is not a geometric property of $M$. A manifold
$M$ may admit different coset descriptions $G/H$ and $G'/H'$, one of which is reductive and the other one is not.}
\end{remark}  

Nomizu's Theorem \cite{Draper:teoNomizu} can be formulated as follows:

\begin{theorem}\label{nomizu}
Let $G/H$ be a reductive homogeneous space with a fixed reductive decomposition (\ref{eq_descomposicionreductiva}).
Then, there is a bijective correspondence between the set of   $G$-invariant affine connections $\nabla$ on $G/H$ and the vector space of bilinear maps $\alpha\colon \mathfrak{m}\times\mathfrak{m}\to\mathfrak{m}$ such that $\mathrm{Ad}(H)\subset\mathrm{Aut}(\mathfrak{m},\alpha)$, that is, such that
$
\alpha\big(\Ad(\sigma)(A),\Ad(\sigma)(B)\big)=\Ad(\sigma)(\alpha(A,B))
$ for all $A,B\in\mathfrak{m}$ and $\sigma\in H$. In case $H$ is connected, this equation is equivalent to
\begin{equation}\label{casoconexo}
[h,\alpha(A,B)]=\alpha([h,A],B)+\alpha(A,[h,B])
\end{equation}
for all $A,B\in\mathfrak{m}$ and $h\in\mathfrak{h}$.
\end{theorem}

\begin{remark}
{\rm Therefore, in the connected case, the
      set of   $G$-invariant affine connections $\nabla$ on $G/H$
 is in one-to-one correspondence to the set of
 $\mathfrak{h}$-invariant bilinear maps $\alpha\colon \mathfrak{m}\times\mathfrak{m}\to\mathfrak{m}$
 (i.e.,     $\text{ad}( \mathfrak{h})\subset\mathfrak{der}(\mathfrak{m},\alpha)$).
 This set is also in bijective correspondence with the vector space of the homomorphisms of $\mathfrak{h}$-modules 
 \begin{equation*} 
 \hom_{\mathfrak{h}}(\mathfrak{m}\otimes\mathfrak{m},\mathfrak{m}),
 \end{equation*}
 through the map which sends  each bilinear map $\alpha$   to   the $\mathfrak{h}$-module homomorphism $\widetilde\alpha$ given by $\widetilde\alpha(X\otimes Y) =\alpha(X,Y)$ for all $X,Y\in\mathfrak{m}$.}
\end{remark}

In order to recover a $G$-invariant affine connection $\nabla$ from the bilinear map $\alpha_{_\nabla}$ attached to $\nabla$ through Theorem~\ref{nomizu}, we will give several explanations on the geometrical meaning of the maps $\alpha$'s in Theorem~\ref{nomizu}. Firstly, recall that for every affine connection $\nabla$ on an arbitrary manifold $M$ and any field $Z\in \mathfrak{X}(M)$,
the \emph{Nomizu operator}  $L^{\nabla}_{Z}$ related to $\nabla$ and   $Z$ is the $(1,1)$-tensor field on $M$ given by
\begin{equation}\label{OpNo}
L^{\nabla}_{Z}X:=[Z,X]-\nabla_{Z}X,
\end{equation}
for all $X\in \mathfrak{X}(M)$. Under the assumptions that  $M$ is $G$-homogeneous and $\nabla$ is $G$-invariant, the following diagram commutes for every $p=\sigma \cdot o\in M$ and $A\in \mathfrak{g}$,

\begin{equation}\label{diagramaA}
\begin{CD}
T_{o} M @>L_{{(\mathrm{Ad}_{\sigma^{-1}}A)}^{+}}^{\nabla}>> T_{o} M\\
@A(\tau_{\sigma^{-1}})_{*}   AA @AA(\tau_{\sigma^{-1}})_{*}A \\
T_{p}M @>>L^{\nabla}_{{A^+}}> T_{p}M.
\end{CD}
\end{equation}
Assume now $M=G/H$ is a reductive homogeneous space with a fixed reductive decomposition  (\ref{eq_descomposicionreductiva}). Then, for $A\in \mathfrak{m}$ the linear map $\alpha_{_\nabla}(A,-)$ corresponding to the $G$-invariant connection $\nabla$ is determined as the unique map such that the following diagram commutes

\begin{equation}\label{diagramaB}  
\begin{CD}
T_{o} M @>L_{A_o^{+}}^{\nabla}>> T_{o} M\\
@A\pi_{*}AA @AA\pi_{*}A \\
\mathfrak{m} @>>\alpha_{_\nabla}(A,-)> \mathfrak{m}.
\end{CD}
\end{equation}   
Thus, for every $A,B\in \mathfrak{m}$ we obtain
\begin{equation}\label{con}
\nabla_{A^{+}_{o}}B^{+}=[A^{+},B^{+}]_{o}-\pi_{*}(\alpha_{_\nabla}(A,B)).
\end{equation}
At every point $p=\sigma \cdot o $, taking into account the above  two diagrams (\ref{diagramaA}) and (\ref{diagramaB}),
we can write in a concise way,
\begin{equation}\label{con2}
\nabla_{A^{+}_{p}}B^{+}=[A^{+},B^{+}]_{p}-(\tau_{\sigma})_{*}\Big(\pi_*\big( \alpha_{_{\nabla}}(\mathrm{Ad}_{\sigma^{-1}}A,
\mathrm{Ad}_{\sigma^{-1}}B)\big)\Big),
\end{equation}
which allows to recover $\nabla$ from $\alpha_{_{\nabla}}$.

\vspace{0.2cm}

The torsion and curvature tensors of the $G$-invariant affine connection $\nabla$ corresponding to the bilinear map $\alpha=\alpha_{_{\nabla}}$ are also computed in \cite{Draper:teoNomizu} as follows:
\begin{equation}\label{tor}
T^{\nabla}(A,B)=\alpha (A,B)-\alpha (B,A)-[A,B]_\mathfrak{m},
\end{equation}
\begin{equation}\label{cur}
 R^{\nabla}(A,B)C=
 \alpha (A,\alpha (B, C))-\alpha (B,\alpha (A, C))-\alpha ( [A , B]_{\mathfrak{m}}, C)-[[A , B]_{\mathfrak{h}} , C],
 \end{equation}
 for any $A, B, C\in \mathfrak{m}$,
where   $[\ ,\ ]_{\mathfrak{h}}$ and   $[\ ,\ ]_{\mathfrak{m}}$ denote the composition of the bracket
($ [\mathfrak{m},\mathfrak{m}]\subset \mathfrak{g}$) with  the projections $\pi_{\mathfrak{h}}$ and $\pi_{\mathfrak{m}}$  of $\mathfrak{g}= \mathfrak {h}\oplus \mathfrak{m}$
 on each factor.

 \begin{remark}
 {\rm 
 Assume $g$ is a $G$-invariant Riemannian metric on $M$.
 Let $\alpha_{_{\nabla}}$ and $\alpha_g$ be the bilinear maps  related (by Nomizu's Theorem) to an invariant affine connection $\nabla$   and to the Levi-Civita connection $\nabla^{g}$, respectively. As a direct consequence of (\ref{con2}), we deduce that $\nabla$ and $\nabla^g$ have the same geodesics if and only if $\alpha_{_{\nabla}}-\alpha_g$ is a skew-symmetric map.}
 \end{remark}

\begin{remark}
 {\rm Note that $\alpha_C (A,B) =0$ and $\alpha_N(A,B)=\frac{1}{2}[A,B]_{\mathfrak{m}}$ correspond to invariant affine connections. They are the \emph{canonical} and the \emph{natural} connections, respectively.
 In the case of symmetric spaces, that is, when $[\mathfrak{m}, \mathfrak{m}]\subset \mathfrak{h}$, both connections are the same.}
 \end{remark}

\vspace{0.2cm}

The fact of a $G$-invariant affine connection to be compatible with a $G$-invariant Riemannian metric $g$ can also be translated to an algebraic setting. Specifically, as a particular case of \cite[Chapter X, Theorem~2.1]{Draper:KobayashiNomizu}, we have the following result. We include here a direct proof for the sake of completeness.

\begin{theorem}\label{th_eldelascompatiblesconlametrica}
Let $M=G/H$ be a reductive homogeneous space  endowed with a $G$-invariant Riemannian metric $g$ and with a fixed reductive decomposition (\ref{eq_descomposicionreductiva}).
 A $G$-invariant affine connection $\nabla$ is metric if and only if the  bilinear operation $\alpha_{_{\nabla}} $ related to $\nabla$ by Theorem~\ref{nomizu}
satisfies
\begin{equation}\label{metric}
  g  (\alpha_{_{\nabla}}(C,A),  B)  + g (A,\alpha_{_{\nabla}} (C,B))=0
\end{equation}
  for any $A,B,C \in \mathfrak{m}$, where $g\colon \mathfrak{m}\times\mathfrak{m}\to\mathfrak{m}$ denotes also the nondegenerate symmetric bilinear map induced by $g$ by means of the identification of $\mathfrak{m}$ with $T_oM$ via $\pi_{*}$.
\end{theorem}
\begin{proof}
Since $g$ is $G$-invariant, we have from (\ref{AD}) that
 \begin{equation}
 g(  A^+_{\sigma\cdot p},  B^+_{\sigma \cdot p}) =g((\mathrm{Ad}_{\sigma ^{-1}} A)^+_p,   (\mathrm{Ad}_{\sigma ^{-1}} B)^+_p),\label{ecutres}
 \end{equation}
for all $\sigma \in G$ and $p\in M$. Recall that $\mathrm{Ad}_{\mathrm{exp}\,(C)}=\mathrm{exp}(\mathrm{ad}_{C})=\sum_{k=0}^{\infty}\frac{(\mathrm{ad}_{C})^{k}}{k!}\in \mathrm{Aut}(\mathfrak{g})$
\cite[3.46]{Warner}.
Taking into account that $A\mapsto A^+$ is an  antihomomorphism of Lie algebras, 
Equation~(\ref{ecutres}) gives
\begin{align*}
 C^+_{p}g(  A^+,  B^+)&=\frac{d}{dt}\vert_{t=0}\Big[ g \Big( A^+_{(\mathrm{exp}\,tC\cdot p)},  B^+_{(\mathrm{exp}\,tC\cdot p)}\Big)\Big]\\[1mm]
&=\frac{d}{dt}\vert_{t=0}\Big[ g \Big( (\mathrm{Ad}_{\mathrm{exp}\,(-tC)}A)^+_{p},(\mathrm{Ad}_{\mathrm{exp}\,(-tC)}B)^+_{p}\Big)\Big]\\[1mm]
&= \frac{d}{dt}\vert_{t=0}\Big[ g \Big( \Big(\sum_{k=0}^{\infty}\frac{(\mathrm{ad}_{-tC})^{k}A}{k!}\Big)^+_{p},\Big(\sum_{k=0}^{\infty}\frac{(\mathrm{ad}_{-tC})^{k}B}{k!}\Big)^+_{p}\Big)\Big]=\\[1mm]
&= -g([C,A]^{+}_{p},B^{+}_{p})-g(A^{+}_{p},[C,B]^{+}_{p})=g([C^+,A^+]_{p},B^{+}_{p})+g(A^{+}_{p},[C^+,B^+]_{p}).
\end{align*}
Thus, $\nabla$ is metric if and only if for every $A,B,C\in \mathfrak{m}$,
\begin{equation}\label{EE}
g(L^{\nabla}_{C^{+}}A^{+}, B^{+})+g(A^{+},L^{\nabla}_{C^{+}}B^{+})=0.
\end{equation}
Evaluating (\ref{EE}) at the origin $o\in M$ and taking into account that $A^{+}_{o}=\pi_{*}(A)$, Equation~(\ref{metric}) easily follows from (\ref{con}).

Conversely, let $p=\sigma \cdot o $ be an arbitrary point of $M$. Then, by applying again that $\tau_{\sigma^{-1}}$ is an isometry jointly with   (\ref{diagramaA}), (\ref{AD}) and (\ref{con}), we get
$$\begin{array}{r}\vspace{2pt}
g(L^{\nabla}_{C^{+}}A^{+}_{p}, B^{+}_{p})=
g \big((\tau_{\sigma^{-1}})_{*}(L^{\nabla}_{C^{+}}A^{+}_{p}), (\tau_{\sigma^{-1}})_{*}(B^{+}_{p})\big)=
g \big(L^{\nabla}_{(\mathrm{Ad}_{\sigma^{-1}}C)^{+}}(\tau_{\sigma^{-1}})_{*}(A^{+}_p), (\mathrm{Ad}_{\sigma^{-1}}B)^{+}_{o}\big)
\\
=g \big(L^{\nabla}_{(\mathrm{Ad}_{\sigma^{-1}}C)^{+}}(\mathrm{Ad}_{\sigma^{-1}}A)^{+}_{o}, (\mathrm{Ad}_{\sigma^{-1}}B)^{+}_{o}\big)=g\big(\alpha_{_\nabla}(\mathrm{Ad}_{\sigma^{-1}}C,\mathrm{Ad}_{\sigma^{-1}}A),\mathrm{Ad}_{\sigma^{-1}}B\big).
\end{array}$$
Finally, from (\ref{metric}) we deduce (\ref{EE}).
\end{proof}

\begin{remark}\label{nnn}
\rm{Equation (\ref{metric}) says that the $\mathfrak{h}$-invariant bilinear map  $\alpha_{_{\nabla}}\colon \mathfrak{m}\times\mathfrak{m}\to\mathfrak{m}$ is related to a $G$-invariant metric affine connection $\nabla$ whenever $\alpha_{_{\nabla}}(X,-)\in\mathfrak{so}(\mathfrak{m},g)$ for all $X\in\mathfrak{m}$. % and conversely.
Therefore, for $H$ connected, there is a bijective correspondence between the set of $G$-invariant affine connections   compatible with the metric $g$ on $M=G/H$ and the vector space $
\hom_{\mathfrak{h}}(\mathfrak{m}, \mathfrak{so}(\mathfrak{m},g)).
$
  But note now that the map
$$\begin{array}{rcl}   
\mathfrak{m}\wedge\mathfrak{m}&\to&\mathfrak{so}(\mathfrak{m},g)\\
x\wedge y&\mapsto& g(x,-)y-g(y,-)x
\end{array}
$$
is an isomorphism of $\mathfrak{h}$-modules  independently of the considered   nondegenerate symmetric  $\mathfrak{h}$-invariant bilinear map $g\colon\mathfrak{m}\times\mathfrak{m}\to\mathfrak{m}$, where as usual $\mathfrak{m}\wedge\mathfrak{m}\equiv\Lambda^2\mathfrak{m}$ denotes  the second exterior power  of the module $\mathfrak{m} $.
Hence there is a one-to-one correspondence between the set of metric $G$-invariant affine connections on $G/H$ and the vector space
$$
\hom_{\mathfrak{h}}(\mathfrak{m},\mathfrak{m}\wedge\mathfrak{m}).
$$
A remarkable consequence of this fact is that the number of free parameters used for the description of the   set of $G$-invariant affine connections   compatible with a $G$-invariant metric $g$ on the homogeneous space $M=G/H$ does not depend on the considered metric. This will allow to extend our results to Berger spheres in a forthcoming paper.   
}
\end{remark}

\vspace{0.10cm}

\section{Odd dimensional spheres}\label{sec_esferasimpares}

Our main purpose here is the study of the homogeneous space $\SS^{2n+1}=\SU(n+1)/\SU(n)$ for every $n\ge1$. In order to fix the notation and conventions used in the sequel, recall that the special unitary group $\SU(n+1)=\{\sigma\in \GL(n+1,\C):\sigma^{-1}=\overline{\sigma}^{t},\,\, \mathrm{det}(\sigma)=1\}$
 acts transitively on the left on the $(2n+1)$-dimensional unit sphere $\SS^{2n+1}\subset \C^{n+1}$ in the natural way by matrix multiplication,
 under the identification of $\C^{n+1}$ with $\R^{2n+2}$ as follows 
 $$
 (z_{1},\hdots,z_{n+1})\leftrightarrow (\mathrm{Re}(z_{1}),\mathrm{Im}(z_{1}),\hdots,\mathrm{Re}(z_{n+1}),\mathrm{Im}(z_{n+1})).
 $$
  The isotropy group for this action at $o:=(0,\hdots,0,1)\in \SS^{2n+1}$ is a closed subgroup of $\SU(n+1)$ isomorphic to
 $\SU(n)$   by identifying $\widetilde{\sigma} \in \SU(n)$ with
 $$
\left(
\begin{array}{c|c}
  \widetilde{\sigma} &0    \\ \hline
  0& 1
\end{array}
\right)\in \SU(n+1).
$$
Thus the sphere $\SS^{2n+1}$ is in a natural way diffeomorphic with the homogeneous manifold $\SU(n+1)/\SU(n)$ \cite[3.65]{Warner}.
The Lie algebra $\mathfrak {su}(n+1)$ consists of the skew-hermitian matrices of zero trace
$$
\mathfrak {su}(n+1)=\{ A\in \mathcal M _{n+1}(\mathbb C) :  A+\bar A^t=0 , \,\, \mathrm{tr}(A)=0\},
$$
and the Lie algebra of the isotropy group at $o$ is isomorphic to $\mathfrak {su}(n)$, seen as a subalgebra of $\mathfrak {su}(n+1)$ by identifying $\widetilde{A} \in \mathfrak {su}(n)$ with
 $$
\left(
\begin{array}{c|c}
  \widetilde{A} &0    \\ \hline
  0& 0
\end{array}
\right)\in \mathfrak {su}(n+1).
$$
Next, consider the vector space decomposition
\begin{equation}\label{ttt}
\mathfrak {su}(n+1)=\mathfrak {su}(n) \oplus\mathfrak{m},
\end{equation}
where
\begin{equation}\label{lam}
\mathfrak{m}=\left\{A=
\left(
\begin{array}{c|c}
 -\frac{a}{n} I_{n} &z    \\ \hline
  -\bar z^t& a
\end{array}
\right) \in \mathcal M_{n+1}(\mathbb C): z^{t}=(z_{1},...,z_{n})\in \mathbb C^n, \,\, a\in \mathbf{i}\R
\right\}.
\end{equation}
Thus $A\leftrightarrow (z,a)$ is a linear isomorphism identifying $\mathfrak{m}$ with $\C^{n}\oplus \mathbf{i}\R$.
This algebraical identification has geometrical meaning. 
The tangent vector space $T_{o}\SS^{2n+1}=\{(z_{1},...,z_{n+1})\in \C^{n+1}: \mathrm{Re}(z_{n+1})=0\} =\C^{n}\oplus \mathbf{i}\R$ and the identification under $\pi_{*}\colon \mathfrak{m}\to T_{o}\SS^{2n+1}$ is also $A\leftrightarrow (z,a)$.  
In what follows, we use this identification to describe the elements in $\mathfrak{m}$.

A direct computation shows that (\ref{ttt}) is a reductive decomposition of $\mathfrak{su}(n+1)$ and under our identification  $A\leftrightarrow (z,a)$, the action of $\mathfrak{su}(n)$ on $\mathfrak{m}$ is given by $B\cdot(z,a)=(Bz,0)$ for $B\in \mathfrak{su}(n)$. Note also
 that the corresponding projection $\pi_{\mathfrak{m}}$ on the factor $\mathfrak{m}$  is given by
\begin{equation*} 
\left(
\begin{array}{c|c}
  B & z    \\ \hline
  -\bar z^t & a
\end{array}
\right)\in \mathfrak{su}(n+1)\xrightarrow{\ \pi_{\mathfrak{m}}\ }  
\left(
\begin{array}{c|c}
 -\frac{a}{n} I_{n}  & z    \\ \hline
  -\bar z^t & a
\end{array}
\right)\in \mathfrak{m}.
\end{equation*}

\smallskip
Consider the decomposition of $\mathfrak{m}$ into a direct sum of $\mathfrak{su}(n)$-irreducible submodules $\mathfrak{m}=\mathfrak{m}_{1}\oplus \mathfrak{m}_{2}$,
 for $\mathfrak{m}_{1}:=\{A\in \mathfrak{m}:a=0\}$ and $\mathfrak{m}_{2}:=\{A\in \mathfrak{m}:z=0\}$. The module $\mathfrak{m}_{1}$ is isomorphic to the natural $\mathfrak{su}(n)$-representation $\mathbb{C}^n$
 and $\mathfrak{m}_{2}$ is a trivial (one-dimensional) module.
 For a geometric interpretation of $\mathfrak{m}_{1}$ and $\mathfrak{m}_{2}$,
recall that the Hopf map is given by $$  \SS^{2n+1}\to \C P^{n},\quad (z_{1},...,z_{n+1})\mapsto [z_{1}:... :z_{n+1}],$$ where $[z_{1}:... :z_{n+1}]\in \C P^n$ represents the complex one-dimensional subspace of $\C^{n+1}$ spanned by $(z_{1},...,z_{n+1})\in \SS^{2n+1}$. The Hopf map is a principal fibre bundle with structural group $\SS^{1}$. 
Let  
$\xi\in \mathfrak{X}(\SS^{2n+1})$ be the vector field given by
\begin{equation}\label{elcampo}
\xi_{z}=-\mathbf{i}z
\end{equation}
at any $z\in \SS^{2n+1}$.
Let $\mathcal{V}=\textrm{Span}(\xi)$ and $\mathcal{H}=\xi^{\perp}$ be the vertical and horizontal distributions for the canonical connection of the Hopf map. The corresponding connection form $\omega$ on $\SS^{2n+1}$ with values in $\mathfrak{s}_{1}=\mathbf{i}\R$ (the Lie algebra of $\SS^1$) satisfies
$
\omega(X)=-\mathbf{i}\,g(X,\xi)
$
for all $X\in \mathfrak{X}(\SS^{2n+1})$.
Now the map $\pi_{*}\colon \mathfrak{m}\to T_{o}\SS^{2n+1} $ allows us to identify $\pi_{*}(\mathfrak{m}_{1})=\mathcal{H}(o)$ and $\pi_{*}(\mathfrak{m}_{2})=\mathcal{V}(o)$. Consistently with the above properties, $\mathfrak{m}_{1}$ and $\mathfrak{m}_{2}$ are called the horizontal and the vertical parts of $\mathfrak{m}$, respectively.
\smallskip

Let us consider $\SS^{2n+1}$ equipped with the canonical Riemannian metric $g$ of constant sectional curvature $1$. The metric $g$ is $\SU(n+1)$-invariant. 
Requiring $\pi_{*}\colon \mathfrak{m}\to T_{o}\SS^{2n+1}$ to be an isometry, $\mathfrak{m}$ is endowed with the inner product (also denoted by $g$) given by
$$
g((z,a),(w,b))=\mathrm{Re}(z^{t}\overline{w})-ab
$$
for any $z,w\in \C^n$ and $a,b\in \bf{i}\R$.\smallskip

In order to derive explicit expressions of the $\SU(n+1)$-invariant affine connections, we summarize several facts on the Sasakian  structure on $\SS^{2n+1}$ (see details in  \cite{Blair}).  For any vector field $X\in \mathfrak{X}(\SS^{2n+1})$, the decomposition of $\mathbf{i}  X$ in tangent and normal components determines the $(1,1)$-tensor field $\psi$ and the $1$-differential form $\eta$ on $\SS^{2n+1}$ such that
\begin{equation}\label{sasaki}
\mathbf{i}  X=\psi(X)+\eta(X)N,
\end{equation}
where $N$ is the unit outward normal vector field to $\SS^{2n+1}$.
Thus, if we denote by $\nabla^g$ the Levi-Civita connection of $g$, the following properties hold
\begin{equation}\label{eq_propiedadesSasakiana}
\begin{array}{ll}
\eta(X)=g(X,\xi),\qquad  &g(\psi(X),\psi(Y))=g(X,Y)-\eta(X)\eta(Y),
\\
\psi^2=-\mathrm{Id}+\eta \otimes \xi,\quad\qquad&(\nabla^{g}_{X}\psi)Y=g(X,Y)\xi-\eta(Y)X,
\end{array}
\end{equation}
for any $X,Y\in \mathfrak{X}(\SS^{2n+1})$. These conditions imply several further relations, for instance we also have that $\eta \circ \psi=0$ and $\psi(\xi)=0$. The Sasakian form is the $2$-differential form $\Phi$ defined by
$\Phi(X,Y)=g(X, \psi(Y))$. Moreover, the vector field $\xi$ is Killing (i.e., $g(\nabla^{g}_{X}\xi,Y )+g(\nabla^{g}_{Y}\xi,X )=0$)
and therefore $\nabla^{g}_{X}\xi=-\psi(X)$ and $2\Phi =d\eta$.

\begin{lemma}\label{blanco}
The Sasakian  structure on $\SS^{2n+1}$ is $\SU(n+1)$-invariant in the following sense. For every $\sigma\in \SU(n+1)$ and $X\in \mathfrak{X}(\SS^{2n+1})$,
$$
\xi=\tau_{\sigma}(\xi),\quad \eta(X)\circ \tau_{\sigma^{-1}}=\eta(\tau_{\sigma}(X)),\quad \tau_{\sigma}(\psi(X))=\psi(\tau_{\sigma}(X)), \quad (\tau_{\sigma})^{*}(\Phi)=\Phi.
$$
\end{lemma}
\begin{proof}
The first assertion is a direct computation. The second one is   consequence of the expression $\eta=g(-,\xi)$ and from the fact that the maps $\tau_\sigma $ are isometries for $g$. Now taking into account that   the Levi-Civita connection $\nabla^{g}$ is $\SU(n+1)$-invariant, we get
$$
\tau_{\sigma}(\psi(X))=-\tau_{\sigma}(\nabla^{g}_{X}\xi)=-\nabla^{g}_{\tau_{\sigma}(X)}\tau_{\sigma}(\xi)=-\nabla^{g}_{\tau_{\sigma}(X)}\xi=\psi(\tau_{\sigma}(X)).
$$
Finally, for every $X,Y\in \mathfrak{X}(\SS^{2n+1})$ and $p\in \SS^{2n+1}$, we have
$$
\begin{array}{ll}\vspace{3pt}
(\tau_{\sigma})^{*}(\Phi)(X_{p},Y_{p})&=g\big((\tau_{\sigma})_{*}(X_{p}), \psi((\tau_{\sigma})_{*}(Y_{p}))\big)=g\big((\tau_{\sigma})_{*}(X_{p}), \psi((\tau_{\sigma})(Y)_{\sigma\cdot p})\big)
\\
 &=g\big((\tau_{\sigma})_{*}(X_{p}), \tau_{\sigma}(\psi (Y))_{\sigma\cdot p}\big)=g\big((\tau_{\sigma})_{*}(X_{p}), (\tau_{\sigma})_{*}(\psi (Y)_{ p})\big)=\Phi(X_{p},Y_{p}).
\end{array}
$$
\end{proof}

  %%%%%%%%%%%%%%%%%%%%%%%%%%%%%%%%%%%%%%%%%%%%%%%%%%%%%
  %%%%%%%%%%%%%%%%%%%%%%%%%%%%%%%%%%%%%%%%%%%%%%%%%%%%%%

\section{Invariant affine connections on $\SS^{2n+1}$  for $n\geq 4$ }

In the low-dimensional cases, the number of $\SU(n+1)$-invariant affine connections on the sphere $\SS^{2n+1}$ depends on $n$. In fact,   we will show that the behavior of the cases $n\geq 4$ and     $\SS^{3}$, $\SS^{5}$ and $\SS^{7}$ is quite different. For this reason, we study  separately each of these.

\subsection{Invariant metric affine connections on $\SS^{2n+1}$}

From now on,  $\mathfrak{m}$ will always be  given by Equation~(\ref{lam}).
We would like to compute the number of (independent) parameters involved in the description of the invariant affine connections. This is a purely algebraic computation: as recalled in Theorem~\ref{nomizu}, the number of parameters coincides with $\mathrm{dim}_{\R}(\mathrm{Hom}_{\mathfrak{su}(n)}(\mathfrak{m}\otimes\mathfrak{m},\mathfrak{m}))$.
Standard arguments of representation theory allow to compute such dimension from the complexification. This forces us to recall some facts about   representations (consult \cite{RepresentacionesReales} for more information).

First, keep in mind that $\mathfrak{m}_1^\mathbb{C}=\mathfrak{m}_1\otimes_{\mathbb{R}}\mathbb{C}$ is a (completely reducible) module for
$\mathfrak{su}(n)^\mathbb{C}=\mathfrak{su}(n)\oplus  {\bf i}\mathfrak{su}(n)=\mathfrak{sl}(n,\mathbb{C})$ under the action $(x+{\bf i}y)(u\otimes(s+{\bf i} t))=(xs-yt)u\otimes1+(xt+ys)u\otimes{\bf i}$ for $x,y\in\mathfrak{su}(n)$, $u\in\mathfrak{m}_1$ and $s,t\in\mathbb{R}$.
Its decomposition as a sum of
   $\mathfrak{sl}(n,\mathbb{C})$-irreducible modules is $V\oplus V^*$, for $V=\mathbb{C}^n$ the $\mathfrak{sl}(n,\mathbb{C})$-natural module and $V^*$ its dual one.
   Indeed, $\mathfrak{m}_1=\mathbb{C}^n$, and $\mathfrak{m}_1^\mathbb{C}=\mathcal{U}_1\oplus\mathcal{U}_2$ is the sum of the $\mathfrak{sl}(n,\mathbb{C})$-modules
   $\mathcal{U}_1:=\{u\otimes 1+{\bf i}u\otimes {\bf i}: u\in\mathfrak{m}_1\}$ and
   $\mathcal{U}_2:=\{u\otimes 1-{\bf i}u\otimes {\bf i}: u\in\mathfrak{m}_1\}$. The map $\mathcal{U}_2\to \mathbb{C}^n$ given by $u\otimes 1-{\bf i}u\otimes {\bf i}\mapsto u$ is an isomorphism  of $\mathfrak{sl}(n,\mathbb{C})$-modules, as well as the map
   $\mathcal{U}_1\to (\mathbb{C}^n)^*$ given by $u\otimes 1+{\bf i}u\otimes {\bf i}\mapsto h(-,u)$,
   for $h\colon \mathbb{C}^n\times\mathbb{C}^n\to\mathbb{C}$, $h(u,v)=u^t\bar v$ the usual Hermitian product.

\begin{lemma}\label{Cris}
For every $n\geq 4$, we have  
$$
\mathrm{dim}_{\R}(\mathrm{Hom}_{\mathfrak{su}(n)}(\mathfrak{m}\otimes\mathfrak{m},\mathfrak{m}))=7.
$$
\end{lemma}

\begin{proof}
The required dimension coincides with
  $
  \dim_{\mathbb{C}}\hom_{\mathfrak{su}(n)^{\mathbb{C}}}(\mathfrak{m}^{\mathbb{C}}
\otimes\mathfrak{m}^{\mathbb{C}},
\mathfrak{m}^{\mathbb{C}}),
 $
 where  $\mathfrak{m}^{\mathbb{C}}=\mathfrak{m}\otimes_{\mathbb{R}}\mathbb{C}$ is the complexified module for $\mathfrak{su}(n)^\mathbb{C}=\mathfrak{su}(n)\oplus  {\bf i}\,\mathfrak{su}(n)=\mathfrak{sl}(n,\mathbb{C})$, which is a simple Lie algebra of type $A_{n-1}$. The decomposition of $\mathfrak{m}^{\mathbb{C}}$ into the direct sum of  irreducible submodules comes from
  $\mathfrak{m}_1^\mathbb{C}\cong V\oplus V^*$
  for $V$ the natural $\mathfrak{sl}(n,\mathbb{C})$-module $\mathbb{C}^n$ and $V^*$ its dual,
  jointly with the obvious fact that
    $\mathfrak{m}_2^\mathbb{C}\cong \mathbb{C}$ is a trivial module.
     In order to decompose the tensor product, note that $V\otimes V\cong S^2V\oplus \Lambda^2V$ (denoting respectively the second symmetric and exterior power) and also that $V\otimes V^*\cong \hom(V,V)= \mathfrak{sl}(V)\oplus\mathbb{C}\id_V$. Hence, the decomposition of $\mathfrak{m}^{\mathbb{C}}
\otimes\mathfrak{m}^{\mathbb{C}}$ as a sum of $\mathfrak{sl}(n,\mathbb{C})$-irreducible representations is
\begin{equation}\label{eq_decomposiciondeltensor}
(V\oplus V^*\oplus \mathbb{C})^{\otimes2}\cong S^2V\oplus  \Lambda^2V\oplus S^2 V^* \oplus  \Lambda^2V^*\oplus 2\mathfrak{sl}(V)\oplus 2V\oplus 2V^*\oplus 3\mathbb{C}.
\end{equation}    
Taking into account that neither $S^2V$, nor $\Lambda^2V$,    nor the adjoint module $\mathfrak{sl}(V)$
are isomorphic to $V$ or $V^*$ (under our  assumptions on $n$), we conclude that the only copies of $V$, $V^*$ or $\mathbb{C}$ in the decomposition~(\ref{eq_decomposiciondeltensor}) are the seven last modules, and the result holds.
 \end{proof}

 \begin{remark}\label{eq_enterminosdepesosftales}
  {\rm In terms of the fundamental weights $\lambda_i$'s (notations as in \cite{Draper:Humphreysalg}),
  $\mathfrak{m}^\mathbb{C}\cong V(\lambda_1)\oplus V(\lambda_{n-1})\oplus V(0)$
  and
  $\mathfrak{m}^\mathbb{C}\otimes\mathfrak{m}^\mathbb{C}\cong V(2\lambda_{1})\oplus V(\lambda_{2})\oplus V(2\lambda_{n-1})\oplus V(\lambda_{n-2})\oplus 2V(\lambda_{1}+\lambda_{n-1})\oplus  2V(\lambda_{1})\oplus 2V(\lambda_{n-1})\oplus 3V(0)$ for all $n\ge3$.}
\end{remark}

Let $\Gamma_n$ be the vector space of $\R$-bilinear maps $\alpha\colon \mathfrak{m}\times\mathfrak{m}\to\mathfrak{m}$ such that $\mathrm{ad}( \mathfrak{su}(n))\subset\mathfrak{der}(\mathfrak{m},\alpha)$.

\begin{proposition}\label{base} For $n\geq 4$, an $\R$-bilinear map $\alpha \in \Gamma_n$ if and only if there exist $q_{1},q_{2},q_{3}\in \mathbb{C}$ and $t\in \mathbb{R}$ such that
\begin{equation}\label{eq_losalfascasoSU}
\alpha((z,a),(w,b))=\left( q_{1}bz+q_{2}aw,\, \mathbf{i}  \left(tab+\mathrm{Im}(q_{3}\overline{z}^tw)\right)\right),
\end{equation}
for all $(z,a),(w,b)\in \mathfrak{m}$.
\end{proposition}

\begin{proof}
A direct computation reveals that the following maps are $\mathfrak{su}(n)$-invariant:
\begin{equation}\label{eq_labase}
\begin{array}{ll}
\alpha_{1}((z,a),(w,b))=(bz,0), &\alpha_{\bf{i}}((z,a),(w,b))=({\bf i}\,bz,0),
\\
\beta_{1}((z,a),(w,b))=(aw,0), &\beta_{\bf{i}}((z,a),(w,b))=({\bf i}\,aw,0),
\\
\gamma_{1}((z,a),(w,b))=(0,{\bf i}\,\mathrm{Im}(\overline{z}^t w)), &\gamma_{\bf{i}}((z,a),(w,b))=(0,{\bf i}\,\mathrm{Im}({\bf i}\,\overline{z}^t w))=(0,{\bf i}\,\mathrm{Re}(\overline{z}^t w)),
\\
\delta((z,a),(w,b))=(0,{\bf i}ab).
\end{array}
\end{equation}
As these maps are linearly independent, Lemma~\ref{Cris} implies that they constitute a basis of $\Gamma_n$ if $n\ge4$ (and a basis of some  subspace of $\Gamma_n$ if $n\le3$).
\end{proof}

\begin{proposition}\label{pr_mesa}
An invariant affine connection $\nabla$ on $\SS^{2n+1}$ ($n\geq 4$) is metric if and only if there are $q\in \C$ and $t\in \R$ such that the corresponding $\R$-bilinear map $\alpha_{_{\nabla}}\in \Gamma_n$ satisfies
\begin{equation}\label{mesa}
\alpha_{_\nabla}=\mathrm{Re}(q)(\alpha_{1}-\gamma_{1})+\mathrm{Im}(q)(\alpha_{\bf{i}}+\gamma_{\bf{i}})+t\beta_{1}.
\end{equation}
\end{proposition}

\begin{proof}
A map  $\alpha$ as in (\ref{eq_losalfascasoSU}) belongs to $\mathfrak{so}(\mathfrak{m},g)$ when $t=\mathrm{Im}(q_{2})=0$ and $q_{3}=-\overline{q_{1}}$, so the result holds.
\end{proof}

\begin{remark}\label{conexionnatural}
{\rm The natural connection $\nabla^N$ is determined by $\alpha_{N}=\frac{1}{2}[\,\,,\,\,]_{\mathfrak{m}}$. Thus, for every $(z,a),(w,b)\in \mathfrak{m}$,
$$
\alpha_{N}((z,a),(w,b))=\frac12\left(\Big(\frac{n+1}{n}\Big)(bz-aw), \overline{w}^tz-\overline{z}^tw\right).
$$
Therefore, the natural connection is achieved for $q_{1}=-q_{2}=\frac{n+1}{2n}$, $t=0$ and $q_{3}=-1$ in
(\ref {eq_losalfascasoSU}). In particular $\nabla^{N}$ is not metric, as we expected since   $\SS^{2n+1}=\SU(n+1)/\SU(n)$ is not a naturally reductive homogeneous space for $n\ne1$.}
\end{remark}

\noindent Now, taking into account (\ref{tor}) and (\ref{mesa}), we can compute the torsion of the metric $\SU(n+1)$-invariant affine connections on $\SS^{2n+1}$.
\begin{corollary}\label{torsionn} For $n\geq 4$, the torsion $T^{_{\nabla}}$ of the $\SU(n+1)$-invariant metric affine connection $\nabla$ corresponding with the $\R$-bilinear map $\alpha_{_{\nabla}}\in \Gamma_n$ in {\rm (\ref{mesa})} is characterized by
$$
T^{_{\nabla}}((z,a),(w,b))=\left(\Big(q-t-\frac{n+1}{n}\Big)(bz-aw), (\mathrm{Re}(q)-1)(\overline{w}^tz-\overline{z}^tw)\right),
$$
for any $(z,a),(w,b)\in \mathfrak{m}$.
\end{corollary}
\begin{proof} Let us first recall that $[\,\,,\,\,]_{\mathfrak{m}}=\frac{n+1}{n}(\alpha_{1}-\beta_{1})-2\gamma_{1}$ and so the proof is straightforward.
\end{proof}

\begin{remark}\label{LV}
{\rm Note that also the torsion $T^\nabla$ belongs to $\Gamma_n$, so that it can be expressed in terms of the basis in (\ref{eq_labase}). Namely,
$$
T^{\nabla}=\left(\textrm{Re}(q)-t-\frac{n+1}{n}\right)(\alpha_{1}-\beta_{1})+\textrm{Im}(q)(\alpha_{{\bf i}}-\beta_{{\bf i}})-2(\textrm{Re}(q)-1)\gamma_{1}.
$$
In particular, the Levi-Civita connection of $\mathbb{S}^{2n+1}$ is got by taking $q=1$ and $t=-1/n$ in (\ref{mesa}). Thus,
$
\alpha_g=\alpha_{1}-\gamma_{1}-\frac{1}{n}\beta_{1}.$
}
\end{remark}

\vspace{0.2cm}
The following result shows how the Sasakian structure on $\SS^{2n+1}$ provides the main tools in order to give explicit formulas  for the invariant metric affine connections. There is no confusion to use the same letter to denote the bilinear maps on $\mathfrak{m}$ and the  corresponding ones under $\pi_{*}$ on $T_{o}\SS^{2n+1}$.

\begin{lemma}\label{cocina}
The bilinear maps on $T_{o}\SS^{2n+1}$ corresponding with $\alpha_{1},\alpha_{\bf{i}},\beta_{1},\gamma_{1}$ and $\gamma_{\bf{i}}$ under the identification $\pi_{*}\colon \mathfrak{m}\to T_{o}\SS^{2n+1}$  are respectively  given by
$$
\alpha_{1}(x,y)=-\eta(y)\psi(x),\qquad\alpha_{\bf{i}}(x,y)=\eta(y)(x-\eta(x)\xi_{o}),\qquad
\beta_{1}(x,y)=-\eta(x)\psi(y),
$$
$$
\gamma_{1}(x,y)=\Phi(x,y)\, \xi_{o},\qquad \gamma_{\bf{i}}(x,y)=-g(\psi(x),\psi(y))\,\xi_{o},
$$
for all $x,y \in T_{o}\SS^{2n+1} $.
\end{lemma}
\begin{proof} Let us consider $x,y\in T_{o}\SS^{2n+1}\subset \R^{2n+2}$   given by $x=(x_{1},y_{1},...,x_{n},y_{n},0,s)$ and $y=(u_{1},v_{1},...,u_{n},v_{n},0,t)$, respectively.   %\margen{nose puede mensionar algo de la identificaci\'{o}n?}
The element $A\leftrightarrow (z,a)\in \mathfrak{m}$ corresponding to $x$ satisfies
$z_{j}=x_{j}+\mathbf{i}y_{j}$ for $j=1,...,n$ and $a=\mathbf{i}s$ and analogously for $y$. We give the proof only for $\gamma_{1}$. Similar computations can be applied to the other cases. First, it is easy to check that
$
\gamma_{1}(x,y)=\big(0,...,0,\sum_{j=1}^{n} (x_{j}v_{j}-y_{j}u_{j})\big).
$
In the same manner we can see that
$
\Phi(x,y)\, \xi_{o}=g(x,\psi(y))\,\xi_{o}=\big(0,...,0,-g(x,\psi(y))\big)\in \R^{2n+2}.
$ The proof is completed from $\psi(y)=(-v_{1},u_{1},...,-v_{n},u_{n},0,0)$.
\end{proof}

\begin{theorem}\label{principal4}
For every $\SU(n+1)$-invariant metric affine connection $\nabla$ on $\SS^{2n+1}$ with $n\geq 4$, there are $q\in\C$ and $t\in \R$ such that
$$
\nabla_{X}Y=\nabla^{g}_{X}Y+(\mathrm{Re}(q)-1)(\Phi(X,Y)\,\xi+\eta(Y)\psi(X))+
\mathrm{Im}(q)(\nabla^{g}_{X}\psi)(Y)+(t+1/n)\eta(X)\psi(Y),
$$
for all $X,Y \in \mathfrak{X}(\SS^{2n+1})$.
Moreover, $\nabla$ has totally skew-symmetric torsion if and only if there is $ r \in \R$ such that
\begin{equation}\label{skew}
\nabla_{X}Y=\nabla^{g}_{X}Y+ r \left(\Phi(X,Y)\,\xi-\eta(X)\psi(Y)+\eta(Y)\psi(X)\right).
\end{equation}
\end{theorem}
\begin{proof} Recall that the difference tensor between $\nabla$ and the Levi-Civita connection $\nabla^g$ is defined by  $\mathcal{D}=\nabla-\nabla^{g}=L^{g}-L^{\nabla}$ where $L^{g}$ and $L^{\nabla}$ denote the Nomizu operators (\ref{OpNo}) of $\nabla^g$ and $\nabla$, respectively. Since $\nabla^{g}$ is $\SU(n+1)$-invariant, $\nabla$ is $\SU(n+1)$-invariant if and only if the difference tensor $\mathcal{D}$ is also $\SU(n+1)$-invariant. According to the expression of $\alpha_{_{\nabla}}$ in (\ref{mesa}), and by using Remark~\ref{LV} and Lemma~\ref{cocina},  we get that, for every $x,y\in T_{o}\SS^{2n+1}$,
$$
\begin{array}{rl}
\mathcal{D}(x,y)&=L^{g}(x,y)-L^{\nabla}(x,y)=\alpha_{g}(x,y)-\alpha_{_{\nabla}}(x,y)
\\
&=(\mathrm{Re}(q)-1)(\Phi(x,y)\,\xi_{o}+\eta(y)\psi(x))+
\mathrm{Im}(q)(\nabla^{g}_{x}\psi)(y)+(t+1/n)\eta(x)\psi(y),
\end{array}
$$
where we have extensively used the properties (\ref{eq_propiedadesSasakiana}) of the Sasakian structure on $\SS^{2n+1}$. Taking into account that $\mathcal{D}$ is $\SU(n+1)$-invariant, Lemma \ref{blanco} ends the proof of the first assertion.

For the second one, notice that the connection  $\nabla$ has totally skew-symmetric torsion if and only if the difference tensor $\mathcal{D}$ is skew-symmetric. This clearly forces $\mathrm{Im}(q)=0$ and $\mathrm{Re}(q)=1-t-1/n$. The proof is completed by taking $ r =-t-1/n$.
\end{proof}

\begin{remark}\label{111}
{\rm The torsion tensors of the metric affine connections with totally skew-symmetric torsion in (\ref{skew}) are given by
$$
T^{\nabla}(X,Y)=2 r \big(\Phi(X,Y)\,\xi-\eta(X)\psi(Y)+\eta(Y)\psi(X)\big).
$$
and the related 3-differential form given in (\ref{tos})  coincides, in this case, with
$$
w_{_\nabla}=\frac12 r \,\eta\wedge d\eta.
$$
This form allows to recover the torsion and then, the connection.}
\end{remark}

\begin{example}\label{222}
{\rm At this point, we would like to write down the expressions   of   several metric affine connections on the sphere  $\SS^{2n+1}$.
\begin{itemize}
\item[i)] The \emph{canonical} connection $\nabla^C$ on $\SS^{2n+1}=\SU(n+1)/\SU(n)$ corresponds to $\alpha_C=0$ and therefore,
$$
\nabla^{C}_{X}Y=\nabla^{g}_{X}Y- \Phi(X,Y)\xi-\eta(Y)\psi(X) +\frac1n\eta(X)\psi(Y)
$$
for any $X,Y \in \mathfrak{X}(\SS^{2n+1})$.
\item[ii)]
The generalized \emph{Tanaka}   
metric connection is defined on the class of contact metric manifolds by the formula
$
\nabla^{*}_{X}Y=\nabla^{g}_{X}Y+\eta(X)\psi(Y)-\eta(Y)\nabla^{g}_{X}\xi+(\nabla^{g}_{X}\eta)(Y)\xi,
$
which satisfies $\nabla^{*}\xi=0$ \cite{Blair}.
Since $\SS^{2n+1}$ is a Sasaki manifold, $\nabla^{g}_{X}\eta=\Phi(X,-)$ and the Tanaka connection reduces to
$$
\nabla^{*}_{X}Y=\nabla^{g}_{X}Y+\eta(X)\psi(Y)+\eta(Y)\psi(X)+\Phi(X,Y)\xi.
$$
Theorem \ref{principal4} shows that $\nabla^{*}$ has not totally skew-symmetric torsion.
\item[iii)]
Recall that  every Sasakian manifold $(M^{2n+1},g,\xi,\eta,\psi)$ admits a unique metric connection $\nabla^c$ with totally skew-symmetric torsion and preserving the Sasakian structure, that is, $\nabla^c\xi=\nabla^c\psi=0$. Furthermore, $g(\nabla^c_{X}Y,Z)=g(\nabla^{g}_{X}Y,Z)+(\eta \wedge \Phi) (X,Y,Z)$ for every $X,Y,Z \in \mathfrak{X}(\SS^{2n+1})$ \cite[Theorem~7.1]{FriIva}. This metric affine connection $\nabla^c$ is called the \emph{characteristic} connection. If we particularize to the case of $\SS^{2n+1}$, it is a simple matter to check that the metric affine connection $\nabla^c$ which preserves the Sasakian structure is achieved in (\ref{skew}) by taking $ r =1$.  (See also \cite{otroTanaka} on adapted connections on metric contact manifolds.)
In order to be used in the final table, we will write 
$
T^{c}(X,Y)=2 (\Phi(X,Y)\,\xi-\eta(X)\psi(Y)+\eta(Y)\psi(X) ) 
$
for the torsion tensor of the characteristic connection.
\end{itemize}}
\end{example}

\subsection{$\nabla$-Einstein manifolds}

The following notion   has recently  been  introduced in \cite{AgriFerr}. A Riemann-Cartan manifold $(M,g,\nabla)$
is said to be Einstein with skew-torsion or just $\nabla$\emph{-Einstein} if the metric affine connection $\nabla$ has totally skew-symmetric torsion and satisfies Equation~(\ref{einstein}):
\begin{equation*} 
\mathrm{Sym}(\mathrm{Ric}^{\nabla})=\frac{s^{\nabla}}{\dim M}\, g,
\end{equation*}
where, following \cite{Draper:surveyagricola}, $\mathrm{Sym}(\mathrm{Ric}^{\nabla}) $ denotes the symmetric part of the Ricci curvature tensor of $\nabla$.   % and $s^{\nabla}$ the scalar curvature   of $\nabla$.
 We will also say that the connection \emph{satisfies the $\nabla$-Einstein equation}.
If  $\omega_{_{\nabla}}$ is the $3$-differential form defined in  (\ref{tos}),  $(M,g, \nabla)$ will be called \emph{Einstein with parallel skew-torsion} if in addition $\nabla \omega_{_{\nabla}}=0$ holds.  
Recall that the classical Einstein metrics for a compact   manifold are the critical points of a variational problem on the total scalar curvature \cite[Chapter~4]{Besse}. In a similar way, the metric connections with skew-torsion  such that $(M,g,\nabla)$ is  $\nabla$-Einstein   are the critical points  of a variational problem which involves the scalar curvature of the Levi-Civita connection $\nabla^g$ and the torsion of $\nabla$ (see details in \cite{AgriFerr}).

Let $S\in \mathcal{T}_{0,2}(M)$ be the tensor given at $p\in M$ by
$$
S(X,Y)_p:=\sum_{j=1}^{n}g(T^{\nabla}(e_{j},X_p),T^{\nabla}(e_{j},Y_p)),
$$
where $\{e_{1},...,e_{n}\}$ is an orthonormal basis of $T_pM$ and $X,Y\in \mathfrak{X}(M)$.   Then, recall the following curvature identities \cite[Appendix]{Draper:surveyagricola},
\begin{equation}\label{formulicas}
\mathrm{Sym}(\mathrm{Ric}^{\nabla})=\mathrm{Ric}^{g}-\frac{1}{4}S, \qquad\qquad s^{\nabla}=s^{g}-\frac{3}{2}\|T^{\nabla}\|^2,
\end{equation}
where $\|T^{\nabla}\|^2:=\frac16\sum_{i,j=1}^{n}g(T^{\nabla}(e_{i},e_{j}),T^{\nabla}(e_{i},e_{j}))$.

\begin{corollary}\label{333}
Let $\nabla$ be a $\SU(n+1)$-invariant metric affine connection with totally skew-symmetric torsion given in {\rm (\ref{skew})} on $\SS^{2n+1}$. Then, the symmetric part of the Ricci tensor $\mathrm{Ric}^{\nabla}$ and the scalar curvature $s^{\nabla}$ are given by
$$
\begin{array}{l}\vspace{2pt}
\mathrm{Sym}(\mathrm{Ric}^{\nabla})=2(n- r ^2)\, g-2(n-1) r ^2 \,\eta \otimes \eta, \\
   s^{\nabla}=2n(2n+1)-6n r ^2.
\end{array}
$$
In particular, $(\SS^{2n+1},g,\nabla)$ is not $\nabla$-Einstein for any $\SU(n+1)$-invariant metric affine connection whenever $n\geq 4$ unless $\nabla=\nabla^g$.
\end{corollary}

\begin{proof}
First we compute the tensor $S$   from the expression of $T^\nabla$ given in Remark \ref{111}.
Take $p\in \SS^{2n+1} $ and $\{e_{1},...,e_{2n+1}\}$ any orthonormal basis of $T_p\SS^{2n+1}$ such that $e_{2n+1}=\xi_p$.
For any tangent vectors $x,y\in T_p\SS^{2n+1}$, we have
$$
\begin{array}{rl}\vspace{2pt}
S(x,y)&=    4 r ^2\sum_{j=1}^{2n+1}\big(\Phi(e_{j},x)\Phi(e_{j},y)+(\eta(e_{j}))^2g(\psi(x),\psi(y))\\
& \quad  -\eta(e_{j})\eta(x)g(\psi(e_{j}),\psi(y))-\eta(e_{j})\eta(y)g(\psi(e_{j}),\psi(x))
+\eta(x)\eta(y)g(\psi(e_{j}),\psi(e_{j}))\big).
\end{array}
$$
From the identities $\eta(e_j)=0$  and $\psi(e_j)=e_j$ if $j\ne 2n+1$, $\eta(e_{2n+1})=1$  and $\psi(e_{2n+1})=0$, we get
$$
\begin{array}{rl}\vspace{2pt}
S(x,y)&=   4 r ^2\sum_{j=1}^{2n+1}\big(\Phi(e_{j},x)\Phi(e_{j},y)+(\eta(e_{j}))^2g(\psi(x),\psi(y))
+\eta(x)\eta(y)g(\psi(e_{j}),\psi(e_{j}))\big)
\\
&= 8 r ^2\big(g(\psi(x),\psi(y))+n\,  \eta(x)\eta(y)\big)=8 r ^2\big(g(x,y)+(n-1)\eta(x)\eta(y)\big).
\end{array}
$$
Hence $S=8 r ^2g+8 r ^2(n-1)\eta\otimes\eta$ holds.  A similar computation applies to obtain $\|T^{\nabla}\|^2=4n r ^2$. The announced formulas for $\mathrm{Sym}(\mathrm{Ric}^{\nabla})$ and $s^{\nabla}$ are now deduced from (\ref{formulicas}), because
$\mathrm{Ric}^{g}=2n    g$ and $s^{g}=2n(2n+1)$.

Then Equation~(\ref{einstein}) holds if and only if either $ r =0$ or $\frac{n-1}{2n+1}g=(n-1)\eta\otimes\eta$. The second possibility is ruled out when  $n\ne1$. 
Finally, the choice   $ r =0$  corresponds with the Levi-Civita connection $\nabla^{g}$.
\end{proof}

\begin{remark}
{\rm As  was mentioned in Example~\ref{222}.\,iii), every odd dimensional sphere $\SS^{2n+1}$ admits a characteristic connection $\nabla^c$, achieved for $r=1$ in  (\ref{skew}).
We have $\mathrm{Ric}^{\nabla^c}(\xi,\xi)= 0$, since the characteristic connection satisfies $\nabla^c\xi=0$. Despite of that, $(\SS^{2n+1},g,\nabla^c)$ is not $\nabla^c$-Einstein as Corollary~\ref{333} shows (compare with \cite[Lemma~2.23]{AgriFerr}).}    
\end{remark}

%%%%%%%%%%%%%%%%%%%%%%%%%%%%%%%%%%%%%%%%%%%%%%%%%%%%%%%%%%%%%%%%%%%%%%%%%%%%%%
%%%%%%%%%%%%%%%%%%%%%%%%%%%%%%%%%%%%%%%%%%%%%%%%%%%%%%%%%%%%%%%%%%%%%%%%%%%%%%%

\section{Invariant connections on $\SS^{7}$  }

For spheres of low dimension, the affine connections described in Theorem \ref{principal4} remain  invariant, but the existence of certain particular invariant tensors provides    new invariant connections, some of them with relevant properties. For $\SS^7$, these tensors come from the 3-Sasakian structure described below.

\subsection{Invariant metric connections on $\SS^{7}$}

\begin{lemma}\label{CrisS7}
$
\mathrm{dim}_{\R} \mathrm{Hom}_{\mathfrak{su}(3)}(\mathfrak{m}\otimes\mathfrak{m},\mathfrak{m}) =9.
$
\end{lemma}

\begin{proof}
As in the proof of Lemma~\ref{Cris},    Equation~(\ref{eq_decomposiciondeltensor}) gives the decomposition of $\mathfrak{m}^{\mathbb{C}}
\otimes\mathfrak{m}^{\mathbb{C}}$ as a sum of irreducible $\mathfrak{sl}(3,\mathbb{C})$-representations.
The point is that, for $V$ the natural $\mathfrak{sl}(3,\mathbb{C})$-module $ \mathbb{C}^3$,
the $\mathfrak{sl}(3,\mathbb{C})$-modules
$\Lambda^2V$ and $V^*$ are isomorphic through the map
$$
\begin{array}{lcl}
\Lambda^2V&\longrightarrow& V^*\\
x\wedge y&\mapsto&\det(x,y,-),
\end{array}
$$
where $\det\colon \Lambda^3V\to\mathbb{C}$ denotes a nonzero fixed trilinear alternating map (in terms of Remark~\ref{eq_enterminosdepesosftales}, $\Lambda^2V\cong V(\lambda_2)=V(\lambda_{n-1})\cong V^*$). The same fact  occurs with their dual modules.
Thus, the number of copies of $V$, $V^*$ and $\mathbb{C}$ in the referred  decomposition~(\ref{eq_decomposiciondeltensor}) is now $9$.
\end{proof}

The above proof gives a hint about how to find the new bilinear maps occurring only for $n=3$. If we denote by $\times$ the cross product in $\mathbb{C}^3$, then the map
$$
\begin{array}{rcl}\mathbb{C}^3\times\mathbb{C}^3&\rightarrow&\mathbb{C}^3 \\
(z,w)&\mapsto& \bar z\times \bar w
\end{array}
$$
is an $\mathfrak{su}(3)$-invariant map.
Let $\Gamma_3$ be the vector space of $\R$-bilinear maps $\alpha\colon \mathfrak{m}\times\mathfrak{m}\to\mathfrak{m}$ such that $\mathrm{ad}( \mathfrak{su}(3))\subset\mathfrak{der}(\mathfrak{m},\alpha)$.

\begin{proposition}\label{bases7} An $\R$-bilinear map $\alpha$ belongs to $\Gamma_3$ if and only if there exist $q_{1},q_{2},q_{3},q_{4}\in \mathbb{C}$ and $t\in \mathbb{R}$ such that
\begin{equation}\label{eq_losalfascasoSU4}
\alpha((z,a),(w,b))=\Big( q_{1}bz+q_{2}aw+q_{4}\,\bar z\times \bar w ,\, \mathbf{i}  \left(tab+\mathrm{Im}(q_{3}\overline{z}^tw)\right)\Big),
\end{equation}
for all $(z,a),(w,b)\in \mathfrak{m}$.
\end{proposition}
\begin{proof}
We know that   $B=\{\alpha_{1},\alpha_{\bf{i}},\beta_{1},\beta_{\bf{i}},\gamma_{1},\gamma_{\bf{i}}, \delta \}$
described in Equation~(\ref{eq_labase}) is a  linearly independent set of  $\mathfrak{su}(3)$-invariant bilinear maps.
Observe that the new maps
$$
\varepsilon_1((z,a),(w,b))=(\bar z\times \bar w,0), \qquad   \varepsilon_{\bf{i}}((z,a),(w,b))=({\bf{i}}\,\bar z\times \bar w,0),      \\
$$
both satisfy Equation~(\ref{casoconexo}) and $B\cup\{\varepsilon_1,\varepsilon_{\bf{i}} \}$ follows being a linearly independent set, providing thus a basis of $\Gamma_3$.
\end{proof}

\begin{proposition}
An invariant affine connection $\nabla$ on $\SS^7$ is metric if and only if there are $q_{1},q_{2}\in \C$ and $t\in \R$ such that the corresponding $\R$-bilinear map $\alpha_{_{\nabla}}\in \Gamma_{3}$ satisfies
\begin{equation}\label{mesa7}
\alpha_{_\nabla}=\mathrm{Re}(q_{1})(\alpha_{1}-\gamma_{1})+\mathrm{Im}(q_{1})(\alpha_{\bf{i}}+\gamma_{\bf{i}})+t\beta_{1}+\mathrm{Re}(q_{2})\varepsilon_{1}+\mathrm{Im}(q_{2})\varepsilon_{\bf i}.
\end{equation}
\end{proposition}

\begin{proof}
First note that $ \varepsilon_{1}$ and $ \varepsilon_{\bf{i}}$ satisfy Equation~(\ref{metric}).
Now a map  $\alpha$ as in Equation~(\ref{eq_losalfascasoSU4}) belongs to $\mathfrak{so}(\mathfrak{m},g)$ when $t=\mathrm{Im}(q_{2})=0$ and $q_{3}=-\overline{q_{1}}$, so the result holds.
\end{proof}

The natural connection can be achieved in the same way as in Remark \ref{conexionnatural}. In particular, we have that
$[\,\,,\,\,]_{\mathfrak{m}}=\frac{4}{3}(\alpha_{1}-\beta_{1})-2\gamma_{1}\in\Gamma_3$.

\begin{corollary}\label{torsion7} The torsion $T^{_{\nabla}}$ of the $\SU(4)$-invariant metric affine connection $\nabla$ corresponding with the $\R$-bilinear map $\alpha_{_{\nabla}}\in \Gamma_3$ in {\rm (\ref{mesa7})} is characterized by
$$
T^{_{\nabla}}((z,a),(w,b))=\Big(\Big(q_{1}-t-\frac{4}{3}\Big)(bz-aw)+2q_{2}\, \bar z\times \bar w , (\mathrm{Re}(q_{1})-1)(\overline{w}^tz-\overline{z}^tw)\Big),
$$
for any $(z,a),(w,b)\in \mathfrak{m}$. In particular, the Levi-Civita connection  $\nabla^g$ is achieved for $q_{1}=1$, $q_{2}=0$ and $t=-1/3$.
\end{corollary}
\begin{proof}  
The proof is an easy computation from (\ref{tor}).
\end{proof}

In order to give explicit expressions for the invariant affine connections which appear in $\SS^7$, we have to consider a specific geometrical structure for seven-dimensional manifolds.
Recall (see for instance   \cite[Chapter 14]{Blair})  that a seven-dimensional Riemannian manifold $(M,g)$ is said to be a $3$-\emph{Sasakian manifold} whenever $M$ is endowed with three compatible Sasakian structures $(\xi_{i},\eta_{i}, \psi_{i})$, $i=1,2,3$. That is, the following formulas hold   
$$
[\xi_{1},\xi_{2}]=2\,\xi_{3},\quad [\xi_{2},\xi_{3}]=2\,\xi_{1},\quad [\xi_{3},\xi_{1}]=2\,\xi_{2}
$$
and also
$$
\begin{array}{lll}\vspace{3pt}
\psi_{3}\circ \psi_{2}=-\psi_{1}+\eta_{2}\otimes \xi_{3},& \psi_{2}\circ \psi_{1}=-\psi_{3}+\eta_{1}\otimes \xi_{2}, & \psi_{1}\circ \psi_{3}=-\psi_{2}+\eta_{3}\otimes \xi_{1},
\\
\psi_{2}\circ \psi_{3}=\psi_{1}+\eta_{3}\otimes \xi_{2},& \psi_{1}\circ \psi_{2}=\psi_{3}+\eta_{2}\otimes \xi_{1},&
\psi_{3}\circ \psi_{1}=\psi_{2}+\eta_{1}\otimes \xi_{3}.
\end{array}
$$
The canonical example of a $3$-Sasakian manifold is the sphere $\SS^{7}$ realized as a hypersurface of $\H^2$ as follows
\begin{equation}\label{eq_identificacion}
\SS^{7}\subset \C^{4}\cong\H^{2}, \quad z=(z_{_1},z_{_2},z_{_3},z_{_4})\mapsto (z_{_1}+z_{_2}\,{\bf j},z_{_3}+z_{_4}\,{\bf j}).
\end{equation}
Now, let $N$ be the unit outward normal vector field to $\SS^{7}$ and  $\xi_{1},\xi_{2},\xi_{3}\in \mathfrak{X}(\SS^{7})$ given by $\xi_{1}(z)=-{\bf i}z$, $\xi_{2}(z)=-{\bf j}z$ and
$\xi_{3}(z)=-{\bf k}z$, respectively. (Note that $\xi_1$   was denoted by $\xi$ in Section 3.)
Analogously to (\ref{sasaki}), for any vector field $X\in \mathfrak{X}(\SS^{7})$, the decompositions of $\mathbf{i} X$, $\mathbf{j} X$ and $\mathbf{k} X$ into tangent and normal components determine the $(1,1)$-tensor fields $\psi_{1},\psi_{2}$ and $\psi_3$ and the differential $1$-forms $\eta_{1},\eta_{2}$ and $\eta_{3}$ on $\SS^{7}$, respectively.
We denote by $\mathcal{D}$ the distribution $\mathrm{Span}\{\xi_{1},\xi_{2},\xi_{3}\}$ (there is no confusion with $\mathcal{D}$ denoting the difference tensor) and by $\Phi_s$ the 2-differential form given by $\Phi_s(X,Y)=g(X,\psi_s(Y))$ for each $s=1,2,3$. 
Note that the complex structure $\psi_1$ on
$\xi_{1}^{\perp}$ satisfies $\psi_{1}(\xi_{2})=\xi_{3}$. The $3$-Sasakian structure provides the main tools to describe the new invariant metric affine connections on $\SS^{7}$. In fact, consider the
 $3$-differential form  
 \begin{equation}\label{LaOmega}
 \Omega=\frac{1}{2}(\eta_{2}\wedge d\eta_{2}-\eta_{3}\wedge d\eta_{3})= \eta_{2}\wedge \Phi_{2}-\eta_{3}\wedge \Phi_{3}
 \end{equation}
  on the $3$-Sasakian manifold $\SS^{7}$. Our first aim is to prove that  $\Omega$ is $\SU(4)$-invariant.
  This   is not an easy task.
  Roughly speaking, we will relate the $\SU(4)$-invariance of    $\Omega$ with that one of the Hermitian metric of the projective complex space.

Recall \cite[Chapter~XI, Example~10.5]{Draper:KobayashiNomizu} that the 3-dimensional projective complex space $\C P^{3}$ can be endowed with the usual Fubini-Study metric $g_{_{FS}}$ of constant holomorphic sectional curvature $4$. Thus, the Hopf map $p\colon \SS^{7}\to \C P^{3}$, $z\mapsto p(z)=[z]$ is a Riemannian submersion. It is a well-known fact that $(\C P^{3}, g_{_{FS}})$ is a K\"ahler manifold. We denote by $J$ its complex structure and by $H$ its Hermitian metric, which satisfies
$\mathrm{Re}(H)=g_{_{FS}}$ and $\mathrm{Im}(H)=\omega$, where $\omega$ is the K\"ahler form of  $(\C P^{3}, g_{_{FS}})$.
The Lie group $\SU(4)$ acts transitively on the left on $\C P^{3}$ in  such a way that $p(\tau_{\sigma}(z))=\tau_{\sigma}(p(z))$ for every $\sigma\in \SU(4)$ (here we also denote by $\tau_{\sigma}$ the action of $\SU(4)$ on $\C P^3$). The Hermitian metric  $H$ is $\SU(4)$-invariant. Moreover, the Hopf map $p$ satisfies $J\circ p_{*}= p_{*} \circ\psi_{1}$. That is, the map $p$ is a contact-complex Riemannian submersion in the terminology of \cite[Chapter~4]{Submer}.

A $\psi_1$-complex frame $\mathcal{B}=\{l_1, l_2, l_3\}$ on $\xi_{1}^{\perp}$ at $z\in \SS^7$
(in other words,  
$\mathcal{B}\cup\psi_1(\mathcal{B})$
is a real basis of $\xi_{1}^{\perp}(z)\le T_z\SS^7$)
is said to be \emph{unitary} when 
the set $p_{*}(\mathcal{B})=\{p_{*}(l_1 ), p_{*}(l_2 ), p_{*}(l_3 )\}$ is a unitary frame for $H$ at $[z]\in \C P^3$, that is, if
$
H(p_{*}(l_i),p_{*}( l_j))=\delta_{ij}.
$
(Note that $H(p_{*}(l_i),p_{*}( l_j))=g(l_i, l_j)+\mathbf{i}g(l_i, \psi_{1}(l_j))$.)
From $\mathcal{B}$, we introduce the following complex valued $1$-forms on $T_{z}\SS^7$,
\begin{equation}\label{complexforms}
\omega_s:=H( p_{*}(-), p_{*}(l_s))=l_s ^\flat+{\bf i}\psi_{1}(l_s)^{\flat},
 \end{equation}
  for all $s=1,2,3$. Observe that $\omega_{s}(\xi_{1}(z))=0$ for any $s$.

\begin{lemma}\label{algebra2}
For every $z\in \SS^7$ and $x\in T_{z}\SS^7$ with $x\in \mathcal{D}^{\perp}(z)$ and $g(x,x)=1$,
\begin{enumerate}
\item[a)] The set $\{x, \psi_{2}( x), \xi_{2}(z)\}$ is a
$\psi_1$-complex unitary frame on $\xi_{1}^{\perp}(z)$;
\item[b)] The attached 1-form $\omega_{s}\vert_{\xi_{1}^{\perp}(z)}\colon \xi_{1}^{\perp}(z)\to\C$ is $\C$-linear for any $s=1,2,3$, where $\xi_{1}^{\perp}(z)$ is endowed with the complex structure given by $\psi_1$. That is,
$
\omega_{s}(\psi_{1}(u))=\mathbf{i}\omega_{s}(u)
$
for every $u\in \xi_{1}^{\perp}(z)$;
\item[c)] The $3$-differential form $\Omega$ at $z$ is given by $ \Omega_{z}=-\mathrm{Re}(\omega_1\wedge \omega_2\wedge \omega_3).$
\end{enumerate}
\end{lemma}
\begin{proof}
The formulas  relating the three Sasakian structures on $\SS^7$ imply that the set
\begin{equation}\label{basebonita}
\Big\{x, \psi_{1}(x), \psi_{2}(x), \psi_{3}(x),\xi_{2}(z),\xi_{3}(z),\xi_{1}(z)\Big\}
\end{equation}
is  an orthonormal basis of $T_{z}\SS^{7}$ and, since $\psi_{3}(x)=\psi_{1}(\psi_{2}(x))$ and $\xi_{3}(z)=\psi_{1}(\xi_{2}(z))$, the set $\{x, \psi_{2}( x), \xi_{2}(z)\}$ is a
$\psi_1$-complex unitary frame on $\xi_{1}^{\perp}(z)$.

Now recall that $\psi_1$ is an isometry of $ \xi_{1}^{\perp}(z) $ and $\psi_1^2\vert_{\xi_{1}^{\perp}(z)}=-\id$. Thus, direct computations show that
$$
\begin{array}{l}
\omega_{1}(\psi_{1}(u))=-g(\psi_{1}(x),u)+\mathbf{i}g(x,u)=\mathbf{i}\omega_{1}(u),\\
 \omega_{2}(\psi_{1}(u))=-g(\psi_{3}(x),u)+\mathbf{i}g(\psi_{2}(x),u)=\mathbf{i}\omega_{2}(u),\\
 \omega_{3}(\psi_{1}(u))= -\eta_{3}(u)+\mathbf{i}\eta_{2}(u)=\mathbf{i}\omega_{3}(u),
\end{array}
$$
for every $u\in \xi_{1}^{\perp}(z)$.

Finally, we have at the point $z\in \SS^7$ that
$
\Phi_{2}=\psi_{2}(x)^{\flat}\wedge x^{\flat}+\psi_{1}(x)^{\flat}\wedge \psi_{3}(x)^{\flat}+\eta_{1}\wedge \eta_{3}
$
and
$
\Phi_{3}=\psi_{3}(x)^{\flat}\wedge x^{\flat}+\psi_{2}(x)^{\flat}\wedge \psi_{1}(x)^{\flat}+\eta_{2}\wedge \eta_{1}$.
Hence,
$$
 \Omega_{z}=
  \eta_{2}\wedge\big(\psi_{2}(x)^{\flat}\wedge x^{\flat}+\psi_{1}(x)^{\flat}\wedge \psi_{3}(x)^{\flat}\big)-\eta_{3}\wedge\big(\psi_{3}(x)^{\flat}\wedge x^{\flat}+\psi_{2}(x)^{\flat}\wedge \psi_{1}(x)^{\flat}\big),
$$
which coincides with
$$
- \mathrm{Re}(\omega_1\wedge \omega_2\wedge \omega_3)=
- \mathrm{Re}\big((x^\flat+\mathbf{i}\psi_1(x)^\flat)\wedge(\psi_2(x)^\flat+\mathbf{i}\psi_3(x)^\flat)\wedge(\eta_2+\mathbf{i}\eta_3)\big).
$$
\end{proof}

\begin{remark}
{\rm The above lemma permits to give an useful formula for $\Omega_{o}$ at the point $o=(0,0,0,1)\in \SS^7$. In fact, consider the $\psi_1$-complex unitary frame $\{e_{1}, \psi_{2}( e_{1}), \xi_{2}(o)\}$ at $\xi_{1}^{\perp}(o)$
corresponding to $\{(1,0,0,0), (0,1,0,0), (0,0,1,0)\}$ under the usual identification $T_{o}\SS^7\subset \C^4$. Then, for every $u=(u_{1},u_{2},u_{3},{\bf i} t)\in T_{o}\SS^7\subset \C^4$ with $t\in \R$ we get
$$
\omega_{1}(u)=u_{1},\quad \omega_{2}(u)=u_{2}, \quad \omega_{3}(u)=u_{3}.
$$
Therefore, if we denote by $u^{h}=(u_{1},u_{2},u_{3})\in\C^3$ (the coordinates of the projection of $u$ on $\xi_{1}^{\perp}(o)$, 
that is, the horizontal part of $u$ as in Section~3), 
then
\begin{equation}\label{pajaros}
\Omega_{o}(u,v,w)=-\mathrm{Re}(\mathrm{det}(u^{h},v^{h},w^{h}))
\end{equation}
for every $u,v,w\in T_{o}\SS^7$.
}
\end{remark}

To prove the invariance of $\Omega$ respect $\SU(4) $, let us fix  $\sigma  \in \SU(4)$.
As   $(\tau_{\sigma})_{*}(\mathcal{D}^{\perp}(o))\cap \mathcal{D}^{\perp}(\sigma o)\ne\{0\}$ by a dimension argument,
  there exists   $x\in \mathcal{D}^{\perp}(o)\le T_{o}\SS^7$ such that ${(\tau_{\sigma})}_{*}(x)=y\in \mathcal{D}^{\perp}(\sigma o)$.
  We scale $x$ to get  $g(x,x)=1$.
  From  Lemma~\ref{algebra2}, we can give expressions for $\Omega$ at the required points.

  On one hand, from the basis $\{x, \psi_{2}(x),\xi_{2}(o)\}$, we construct the complex valued $1$-forms $\omega_1$, $\omega_2$ and $\omega_3$ as in (\ref{complexforms}), so that $ \Omega_o=-\mathrm{Re}( w_{1}\wedge w_{2}\wedge w_{3})$.
  Recall that $(\tau_{\sigma})_{*}(\xi_{1}(o))=\xi_{1}(\sigma o)$ and $(\tau_{\sigma})_{*}\circ \psi_{1}=\psi_{1}\circ (\tau_{\sigma})_{*}$  by Lemma~\ref{blanco},
  which implies that $\mathcal{B}=\{y, (\tau_{\sigma})_{*}(\psi_{2}(x)), (\tau_{\sigma})_{*}(\xi_{2}(o))\}$ is a $\psi_{1}$-complex unitary basis of $\xi_{1}^{\perp}(\sigma o)$.
  Let  $\Gamma_{1}, \Gamma_{2}$ and $\Gamma_3$ be
  the attached complex valued $1$-forms. The fact that $p\circ \tau_{\sigma} =\tau_{\sigma}\circ  p$ and the invariance property of $H$ with respect to $\SU(4)$ allow to obtain $\tau_{\sigma}^{*} (\Gamma_{s})=\omega_{s}$ for $s=1,2,3$, so that $\tau_{\sigma}^{*}(\Omega)_{o}=-  \mathrm{Re}( \Gamma_{1}\wedge \Gamma_{2}\wedge \Gamma_{3})$.
On the other hand, the basis $\mathcal{B}'=\{y, \psi_{2}(y), \xi_{2}(\sigma o)\}$ of $\xi_{1}^{\perp}(\sigma o)$ satisfies the hypothesis of
  Lemma~\ref{algebra2}, so that the 3-form at $\sigma o $ is given by
  \begin{equation*}
 \Omega_{\sigma o}=- \mathrm{Re}( \omega^{\sigma}_{1}\wedge \omega^{\sigma}_{2}\wedge \omega^{\sigma}_{3}),
\end{equation*}
  being $\omega^{\sigma}_1, \omega^{\sigma}_2$ and $\omega^{\sigma}_3$ the complex valued $1$-forms related to $\mathcal{B}'$. The aim is to show that $\tau_{\sigma}^{*}(\Omega)_{o}=\Omega_{\sigma o}$. By using
  item b) in Lemma \ref{algebra2}, we know  that the forms $\omega^{\sigma}_s$  are $\C$-linear on $\xi_{1}^{\perp}(\sigma o)$ for $s=1,2,3$.
   Reasoning  as in the proof of item b), the forms $\Gamma_s$    are also $\C$-linear and so we have two ($\psi_1$-)complex 3-forms on $\xi_{1}^{\perp}(\sigma o)$, whose complex dimension is 3.
Therefore  there exists $\delta(\sigma,x)\in \C$ such that
 $$
 \omega^{\sigma}_{1}\wedge \omega^{\sigma}_2 \wedge \omega^{\sigma}_3= \delta(\sigma,x )\, \Gamma_{1}\wedge \Gamma_{2}\wedge \Gamma_{3}.
 $$
 This complex number satisfies $\delta(\sigma,x)\in \SS^1=\{\delta\in\C: \delta\bar \delta\equiv\vert\delta\vert=1\}$, since it is the determinant of a   change of basis between the unitary bases  $\mathcal{B}$ and $\mathcal{B}'$. In order to prove that $\delta(\sigma,x)=1$, consider the usual Hermitian product $h(z,w)=\sum_{l=1}^4 z_{l}\overline{w_l}$ in $\C^4$. Recall that $\mathrm{Re}(h)$ is nothing but $\langle\,\,,\,\, \rangle$, the usual inner product on $\R^8$.
In fact, we have   $h(z,w)=\langle z,w\rangle+{\bf i}\langle z, {\bf i}w\rangle$ for all $z,w\in \C^{4}$.
 A  straightforward computation shows that
 \begin{equation}\label{eq_parainvarianza}
 \begin{array}{rl}\vspace{3pt}
 \delta(\sigma,x)&=(\omega^{\sigma}_{2}\wedge \omega^{\sigma}_3)\big((\tau_{\sigma})_{*}(\psi_{2}(x)), (\tau_{\sigma})_{*}(\xi_{2}(o))\big)\\
 &=
 \det\left(\begin{array}{ cc }
 h(\sigma \mathbf{j} x, {\bf j}\sigma x) &    h(-\sigma {\bf j} o, {\bf j}\sigma x)  \\
    h(\sigma \mathbf{j} x,  -{\bf j}\sigma o) &  h(-\sigma {\bf j} o, -{\bf j}\sigma o)
\end{array}\right),
\end{array}
 \end{equation}
where $x=(x_{1}, x_{2},0,0)\in T_{o}\SS^7\subset \C^{4}\cong \H^2 $ with the above  identification~(\ref{eq_identificacion}).

The problem   is now easily treatable from an algebraical level. Recall that, from that identification~(\ref{eq_identificacion}), the conjugate linear isomorphism of $\C^4$ given by the multiplication by ${\bf j}$ acts as
${\bf j}(z_1,z_2,z_3,z_4)=(-\bar z_2,\bar z_1,-\bar z_4,\bar z_3)$. Moreover, ${\bf j}^2=-\id$ and the relationship with the usual Hermitian product $h$ is the following: $h({\bf j}u,v)=-\overline{h(u,{\bf j}v)}=-h({\bf j}v,u)$ for any $u,v\in \C^4$.

\begin{lemma}\label{le_positivo}
Consider the Grassmannian manifold 
$
\mathrm{G}_{2,2}(\C)
$
of two-dimensional complex subspaces of $\C^4$.
Define the map
$$
\begin{array}{rccl}
\beta\colon& \mathrm{G}_{2,2}(\C)&\to &\mathbb{R}\\
&U&\mapsto&\beta(U)=\det(u_1\vert\, {\bf j}u_1\vert u_2\vert\, {\bf j}u_2),
\end{array}
$$
for any $\{u_1,u_2\}$ unitary basis of $U$. Then:
\begin{itemize}
\item[a)] This map is well defined;
\item[b)] $\beta(U)=0$ if and only if $U= {\bf j}U$;
\item[c)] $\mathcal{A}:=\{U\in \mathrm{G}_{2,2}(\C)\mid  {\bf j}U\ne U\}$ is a connected open set of $\mathrm{G}_{2,2}(\C)$.
\end{itemize}
In particular  $\beta(U)\ge 0$ for any $U\in \mathrm{G}_{2,2}(\C)$.
\end{lemma}

\begin{proof}
For any $u_1,u_2\in\C^4$,  denote $\beta(u_1,u_2):= \det(u_1\vert\, {\bf j}u_1\vert u_2\vert\, {\bf j}u_2)$. Since we have
that
\begin{equation*}\label{eq_porsinounitarialabase}
\beta(\alpha_{11}u_1+\alpha_{12}u_2,\alpha_{21}u_1+\alpha_{22}u_2)=\left\vert \det\left(
\begin{array}{cc}
\alpha_{11}&\alpha_{12}\\
\alpha_{21}&\alpha_{22}
\end{array}\right)\right\vert \beta(u_1,u_2)
\end{equation*}
for any $\alpha_{ij}\in\C$,
we get that $\beta(U)$ does not depend on the choice of the unitary basis of $U$.
The fact of being  $\beta(U)$ a real number is easily deduced from
   $-PCP=\overline{C}$ for $C=(u_1\vert\,{\bf j}u_1\vert\,u_2\vert\,{\bf j}u_2)$ and
$P=\tiny{\left(\begin{array}{rrrr}
0&-1&0&0\\
1&0&0&0\\
0&0&0&-1\\
0&0&1&0
\end{array}
\right)} $. In fact, $\det (-P^2)=\det I_4=1$ implies $\det C=\overline{\det C}$.

Take $U \in \mathrm{G}_{2,2}(\C)$ with $\beta(U)=0$.
Since $\{u_1,{\bf j}u_1,u_2,{\bf j}u_2\}$ is linearly dependent for any unitary basis $\{ u_1,u_2 \}$ of $U$, then $U+{\bf j}U\ne\C^4$ and there exists  $0\ne u\in U\cap {\bf j}U$. We can assume without loss of generality that $u=u_1$.
  Thus ${\bf j}u_1=\alpha_1u_1+\alpha_2u_2$ with $\alpha_2\ne0$, since for any
  $0\ne v\in\C^4$ the set $\{v,{\bf j}v\}$ is linearly independent ($h(v,{\bf j}v)=0\ne h(v,v)$).
  By multiplying by ${\bf j}$, we get ${\bf j}u_2=\frac1{\bar\alpha_2}(-u_1-\bar \alpha_1{\bf j}u_1)\in U$, so that ${\bf j}U=U$. This proves b) since the converse is trivial.

  Since $\beta$ is obviously a continuous map, $\mathcal{A}=\beta^{-1}(\mathbb{R}\setminus\{0\})$ must be an open set.
  In order to check the connectedness of $\mathcal{A}$,  we will provide, for any $U,V\in \mathcal{A}$ with $U\ne V$, a curve $\gamma\colon[0,1]\to \mathrm{G}_{2,2}(\C)$   such that $\gamma(0)=V$, $\gamma(1)=U$ and $\gamma(t)\in\mathcal{A}$ for any $t\in(0,1)$.
  First assume that $U\cap V\ne0$. Thus there are $v\in V$ and a basis $\{u_1,u_2\}$ of $U$ such that $\{u_1,v\}$ is a basis of $V$.
  If ${\bf j}u_1\notin U+V$, take
  $$
  \gamma(t)=\mathrm{Span} \{ u_1,tu_2+(1-t)v\},
  $$
  which satisfies the required conditions.
  If ${\bf j}u_1\in U+V$, complete to a basis  $\{u_1,u_2,v,w\}$  of $\C^4$ and take
  $$
  \gamma(t)=\mathrm{Span}\{  u_1,tu_2+(1-t)v+t(1-t)w\}.
  $$
Second assume that $U+V=\C^4$. Since $\beta(U)\ne0$, then $U+{\bf j}U=\C^4$ and $\pi_V({\bf j}U)=V$, for $\pi_V$ the ($h$-)orthogonal projection on $V$.
Take $\{u_1,u_2\}$ a unitary basis of $U$ and let $v_i=\pi_V({\bf j}u_i)$ for $i=1,2$. Then there is $\alpha\in\C$ such that ${\bf j}u_1=\pi_U({\bf j}u_1)+\pi_V({\bf j}u_1)=\alpha u_2+v_1$.
Note that this implies that ${\bf j}u_2=-\alpha u_1+v_2$, ${\bf j}v_1=(-1+\alpha\bar\alpha) u_1-\bar\alpha v_2$
and  ${\bf j}v_2=(-1+\alpha\bar\alpha) u_2+\bar\alpha v_1$. Thus, a suitable curve is, for instance,
$$
  \gamma(t)=\mathrm{Span}\{  tu_1+(1-t)v_1,tu_2+\varepsilon(1-t)v_2\},
  $$
  being $\varepsilon$ any nonzero (fixed) complex number if $\alpha=0$ and   $\varepsilon\ne-\frac{\bar\alpha}{\alpha}$ if $\alpha\ne0$.

Therefore the map $\beta$ has   constant sign. For $U=\mathrm{Span}\{(1,0,0,0), (0,0,1,0)\}$, it holds $\beta(U)=\det(I_4)=1>0$, what finishes the proof.
\end{proof}

\begin{lemma}\label{algebra}
 For any $U\in \mathrm{G}_{2,2}(\C)$,      
take $V=U^\perp$   its orthogonal  subspace relative to the usual Hermitian product $h$.
 For any $B_U=\{u_1,u_2\}$ and $B_V=\{v_1,v_2\}$ bases of $U$ and $V$ respectively, consider the matrices
 $$
 \sigma_{_{B_U,B_V}}=(u_{1}|v_{1}|u_{2}|v_{2}) \in   \GL(4,\C)
 $$
  and
 $$
 \tilde\sigma_{_{B_U,B_V}}=\left(\begin{array}{ cc }
 h(v_1 , {\bf j}u_{1}) &   h (v_2 , {\bf j}u_{1})  \\
h(v_1 , {\bf j}u_{2}) & h(v_2 , {\bf j}u_{2})
\end{array}\right)\in   \mathcal{M}_2(\C).
$$
Then
  $\alpha(U)= \frac{\det\left(\tilde\sigma_{_{B_U,B_V}}\right)}{\det\left(\sigma_{_{B_U,B_V}}\right)}$  is independent of the choice of the bases $B_U$ and $B_V$.
Furthermore  $\alpha(U)$ is a real nonnegative number.

\end{lemma}
\begin{proof}

First, if we take a new basis $B_V'=\{\alpha_{11}v_1+\alpha_{12}v_2,\alpha_{21}v_1+\alpha_{22}v_2\}$ of $V$, we check that
$$\begin{array}{l}
\det\big(\sigma_{_{B_U,B_V'}}\big)=(\alpha_{11}\alpha_{22}-\alpha_{12}\alpha_{21}) \det\big(\sigma_{_{B_U,B_V}}\big),\\
\det\big(\tilde\sigma_{_{B_U,B_V'}}\big)=(\alpha_{11}\alpha_{22}-\alpha_{12}\alpha_{21}) \det\big(\tilde\sigma_{_{B_U,B_V}}\big),
\end{array}
$$
so   $\alpha(U)$ does not depend on the chosen basis of $V=U^\perp$. (The same reasoning applies to the bases of $U$.)

For checking that  $\alpha(U )$ is real, consider the orthogonal projection $\pi_V\colon\C^4\to V$, which is $\C$-linear, and fix a unitary basis $B_U$  of $U$.
In case that
$\{\pi_V({\bf j}u_1),\pi_V({\bf j}u_2)\}$ is   a $\C$-linearly  dependent set,
    there is a nonzero vector $v\in V$ ($h$-)orthogonal to the subspace $\langle\{\pi_V({\bf j}u_1),\pi_V({\bf j}u_2)\}\rangle$.  
    We choose the basis $B_V$ such that $v_1=v$. Thus the elements in the first column of the matrix
$\tilde\sigma_{_{B_U,B_V}}$ are $h(v,{\bf j}u_s)=h(v,\pi_V({\bf j}u_s))=0$ for all $s=1,2$, so that  $\alpha(U )=0$.

Otherwise, $\{\pi_V({\bf j}u_1),\pi_V({\bf j}u_2)\}$ constitutes a basis $B'_V$ of $V$. On one hand, we have that
 $$
 \begin{array}{rl}
 \det\big(\tilde\sigma_{_{B_U,B'_V}}\big)=&h(\pi_V({\bf j}u_1),\pi_V({\bf j}u_1))h(\pi_V({\bf j}u_2),\pi_V({\bf j}u_2))\\
 &-
 h(\pi_V({\bf j}u_1),\pi_V({\bf j}u_2))h(\pi_V({\bf j}u_2),\pi_V({\bf j}u_1))\in\mathbb{R}_{\ge0},
 \end{array}
 $$
   by the Cauchy-Schwarz inequality. On the other hand, as a determinant with three columns in $U$ is necessarily zero,
$$
\begin{array}{rl}
\det\big( \sigma_{_{B_U,B'_V}}\big)&=\det(u_1\vert \pi_V({\bf j}u_1)\vert u_2\vert \pi_V({\bf j}u_2)) \\
&=\det(u_1\vert\, {\bf j}u_1\vert u_2\vert\, {\bf j}u_2)=\beta(U)\in\mathbb{R}_{\ge0},
\end{array}
$$
by Lemma~\ref{le_positivo}. Then the quotient $\alpha(U)$   is also real
and nonnegative.

\end{proof}

Now we are in a position to show the announced invariance of $\Omega$.

\begin{proposition}\label{invarianzaCris}
The $3$-differential form $\Omega=\frac{1}{2}(\eta_{2}\wedge d\eta_{2}-\eta_{3}\wedge d\eta_{3})$ on the $3$-Sasakian manifold $\SS^{7}$ is $\SU(4)$-invariant. That is, 
$
\tau_{\sigma}^{*}(\Omega)=\Omega 
$
 for every $\sigma\in \SU(4)$.
\end{proposition}

\begin{proof}
Fix $\sigma \in \SU(4)$ and  take   $x\in \mathcal{D}^{\perp}(o)\cap(\tau_{\sigma})_{*}^{-1}(\mathcal{D}^{\perp}(\sigma o)) $ (which forces $x=(x_1,x_2,0,0)$ under the identification $T_o\SS^7\subset\C^4$)   such that $g(x,x)=1$.  
Thus  $\{x, {\bf j}x,o,{\bf j}o\}$ is an orthogonal basis   of $\C^4$ (relative to $h$) and so is   $\{\sigma x ,\sigma o,\sigma {\bf j}x ,\sigma {\bf j}o \}$. We can apply  Lemma~\ref{algebra} to $B_U= \{\sigma x ,\sigma o   \} $ and $B_V= \{\sigma {\bf j}x ,\sigma {\bf j}o   \} $.
Observe that $\delta(\sigma,x)$ coincides, according to   Equation~(\ref{eq_parainvarianza}),   with
$$
\delta(\sigma,x)=\det\big(\tilde \sigma_{_{B_U,B_V}}\big)=\alpha(U)\det\big(  \sigma_{_{B_U,B_V}}\big),
$$
but
$$
\det\big(  \sigma_{_{B_U,B_V}}\big)=\det ( \sigma x\vert\,  \sigma{\bf j}x\vert\, \sigma o\vert\, \sigma{\bf j}o)=\det \sigma \det\left(
\begin{array}{cccc}
x_1&-\bar x_2&0&0\\
x_2&\bar x_1&0&0\\
0&0&0&-1\\
0&0&1&0
\end{array}
\right)=g(x,x)=1.
$$
Hence $ \delta(\sigma,x)=\alpha(U)$ is a nonnegative real number which belongs to $\SS^1$, that is, $\delta(\sigma,x)=1$. Hence
$\tau_{\sigma}^{*}(\Omega)_{o}=\Omega_{\sigma o}$ for every  $\sigma \in \SU(4)$.
Now the transitivity of the $\SU(4)$-action allows us to conclude that $\tau_{\sigma}^{*}(\Omega)_{z}=\Omega_{\sigma z}$ for any $z\in\SS^7$.

\end{proof}

\begin{remark}
{\rm The $3$-differential form $\Omega$ on $\SS^7$ can be thought as an extension of the determinant on $\xi_{1}^{\perp}(o)\cong \C^3$ to all the distribution $\xi_{1}^{\perp}$ (see (\ref{pajaros})).
}
\end{remark}

From $\Omega$ we introduce the following tensor  fields on $\SS^7$. For every $X,Y, Z\in \mathfrak{X}(\SS^7)$, the vector fields $\Theta (X,Y)$ and $\widetilde{\Theta}(X,Y)$ and the $(1,1)$-tensor fields $\Theta_{Y}$ and
$\widetilde{\Theta}_Y$ are given by
$$
g(\Theta (X,Y),Z)=g(\Theta_{Y}X,Z)=\Omega(X,Y,Z)
$$
and
$$
g(\widetilde{\Theta} (X,Y),Z)=g(\widetilde{\Theta}_{Y} X,Z)=\Omega(X,Y,\psi_{1}(Z)).
$$
Observe that $\widetilde{\Theta}(X,Y)=-\psi_{1}(\Theta (X,Y)).$
As a direct consequence of Proposition~\ref{invarianzaCris}, we obtain that
\begin{equation}\label{eq_lainvarianza cason3}
 \tau_{\sigma} (\Theta(X,Y))=\Theta ( \tau_{\sigma} (X), \tau_{\sigma} (Y))
%(\tau_{\sigma})_{*}(\Theta(X,Y))=\Theta ((\tau_{\sigma})_{*}(X),(\tau_{\sigma})_{*}(Y))
\end{equation}
  for all $X,Y\in \mathfrak{X}(\SS^7)$ and $\sigma\in\SU(4)$. Analogously   $\widetilde{\Theta}$ is also $\SU(4)$-invariant.
  For every $z\in \SS^7$ and every orthonormal basis  of $T_{z}\SS^7$  as in (\ref{basebonita}), a direct computation gives
  that the  operation $  \Theta_z$ on $T_{z}\SS^7 $ is given by $ \Theta_z(\xi_{1}(z),v)=0 $ for all $v\in T_{z}\SS^7$ and by the following table.

\begin{table}[htdp]\label{tabla}
\centering
\begin{tabular}{ c||c|c|c|c|c|c|}
$\Theta_z $ & $x$ & $\psi_{1}(x)$ & $\psi_{2}(x)$ & $\psi_{3}(x)$ & $\xi_{2}(z)$ & $\xi_{3}(z)$ \\
\hline\hline
$ x $ & $0$ & $0$ & $ -\xi_{2}(z)$ & $ \xi_{3}(z)$ & $ \psi_{2}(x) $ & $ - \psi_{3}(x)$ \\
\hline
$ \psi_{1}(x) $ & $0$ & $0$ & $\xi_{3}(z) $ & $ \xi_{2}(z)$ & $- \psi_{3}(x) $ & $- \psi_{2}(x) $ \\
\hline
$ \psi_{2}(x) $ & $ \xi_{2}(z)$ & $-\xi_{3}(z)  $ & $ 0$ & $0 $ & $-  x $ & $ \psi_{1}(x) $ \\
\hline
$ \psi_{3}(x) $ & $ -\xi_{3}(z)$ & $-\xi_{2}(z) $ & $0$ & $0$ & $ \psi_{1}(x)$ & $ x $ \\
\hline
$ \xi_{2}(z) $ & $ - \psi_{2}(x)$ & $ \psi_{3}(x) $ & $   x$ & $-  \psi_{1}(x) $ & $0$ & $0$ \\
\hline
$ \xi_{3}(z) $ & $  \psi_{3}(x)$ & $  \psi_{2}(x)$ & $- \psi_{1}(x) $ & $ - x$ & $0$ & $0$ \\
\hline
\end{tabular}\vspace{3pt}
\caption{Operation $  \Theta$}\label{tablatita}
\end{table}

\begin{remark}\label{operaciontheta}
{\rm Let $V$ be a finite dimensional real vector space $V$ endowed with a nondegenerate bilinear form $g$. Recall that a $2$-fold \emph{vector cross product} on $V$ is a bilinear map $P\colon V\times V\to V$ satisfying
$$
g(P(x,y),x)=g(P(x,y),y)=0
$$
and
$$
g(P(x,y), P(x,y))=\det\left(\begin{array}{cc}
 g(x,x) &    g(x,y)  \\
   g(y,x) &  g(y,y)
\end{array}\right)
$$
 for every $x,y \in V$  \cite{Gray}.
The operation $\Theta$ satisfies the first axiom of a $2$-fold vector cross on $\xi_{1}^{\perp}(z)$ but not the second one.}
\end{remark}

Finally, we introduce the $(0,2)$-tensor field $B$ on $\SS^7$ as follows,
$$
B(X,Y)=\mathrm{tr}(\Theta_{X}\circ \Theta_{Y}).
$$
It is obvious that $B$ is a symmetric tensor by taking into account that $g(\Theta_XY,Z)+g(Y,\Theta_XZ)=0$ for all $X,Y,Z\in\mathfrak{X}(\SS^7)$.
For every orthonormal basis $\mathcal{B} $ as in (\ref{basebonita}), we have $B(u,u)=-4$ for all $u \in \mathcal{B}\setminus\{\xi_{1}(z)\}$,
also $B(\xi_{1}(z) ,\xi_{1}(z) )=0 $
and    $B(u,v)=0$ for all   $u\ne v$ elements in $\mathcal{B}$. Thus, we get
\begin{equation}\label{tensorB}
B=4(\eta_{1} \otimes \eta_{1} -g).
\end{equation}

Now we will see how  the $3$-differential form $\Omega$ allows us to give explicit formulas for the new invariant connections coming from $\varepsilon_{1}$ and $\varepsilon_{\bf i}$.

\begin{lemma}\label{cocina7}
The bilinear maps on $T_{o}\SS^{7}$ corresponding with $\varepsilon_{1}$ and $\varepsilon_{\bf{i}}$ under the identification $\pi_{*}\colon \mathfrak{m}\to T_{o}\SS^{7}$  are given by
$$
\varepsilon_{1}((z,a),(w,b))=- \Theta_o((z,a),(w,b)),\qquad \varepsilon_{{\bf i}}((z,a),(w,b))= \widetilde{\Theta}_o ((z,a),(w,b)),
$$
for all $(z,a), (w,b) \in T_{o}\SS^{7} $.
\end{lemma}
\begin{proof} A direct computation taking into account Equation~(\ref{pajaros}) shows that
$$
\begin{array}{ll}
g\left(\varepsilon_{1}((z,a),(w,b)),(u,c)\right)&=\mathrm{Re}\left(h(\bar{z}\times \bar{w}, u)\right)=\mathrm{Re}\,(\overline{\mathrm{det}(z,w,u)}) \\
&=
\mathrm{Re}\,(\mathrm{det}(z,w,u))= - \Omega_{o}((z,a),(w,b),(u,c)).
\end{array}
$$
A similar argument works for $\varepsilon_{\bf{i}}$.  
\end{proof}

\begin{theorem}\label{principal7}
For every $\SU(4)$-invariant metric affine connection $\nabla$ on $\SS^{7}$, there are $q_{1},q_{2}\in\C$ and $t\in \R$ such that
$$
\begin{array}{ll}\vspace{3pt}
\nabla_{X}Y=&\nabla^{g}_{X}Y+(\mathrm{Re}(q_{1})-1)(\Phi_{1}(X,Y)\,\xi_{1}+\eta_{1}(Y)\psi_{1}(X))+
\mathrm{Im}(q_{1})(\nabla^{g}_{X}\psi_{1})(Y)\\
&+(t+\frac13)\eta_{1}(X)\psi_{1}(Y)-
\mathrm{Re}(q_{2})\Theta(X,Y)+\mathrm{Im}(q_{2})\widetilde{\Theta}(X,Y)
\end{array}
$$
for all $X,Y \in \mathfrak{X}(\SS^{7})$.
Moreover, $\nabla$ has totally skew-symmetric torsion if and only if  there are $ r \in \R$ and $q\in \C$ such that
\begin{equation}\label{skew7}
\begin{array}{ll}\vspace{3pt}
\nabla_{X}Y=&\nabla^{g}_{X}Y+ r \Big(\Phi_{1}(X,Y)\,\xi_{1}-\eta_{1}(X)\psi_{1}(Y)+\eta_{1}(Y)\psi_{1}(X)\Big)\\
&+ \mathrm{Re}(q)\Theta(X,Y)+\mathrm{Im}(q)\widetilde{\Theta}(X,Y).
 \end{array}
\end{equation}
\end{theorem}
\begin{proof} The proof closely follows that of  Theorem~\ref{principal4} by using Equation~(\ref{mesa7}) for the expression of $\alpha_{_\nabla}$,  
 Lemmas~\ref{cocina} and \ref{cocina7} for writing    $\alpha_{_\nabla}$ in terms of the 3-Sasakian structure, and then the invariance property  in Lemmas~\ref{blanco} and Equation~(\ref{eq_lainvarianza cason3}) for extending the $\SU(4)$-invariant difference tensor $\mathcal{D}=\nabla-\nabla^g$.

In order to obtain the invariant connections with totally skew-symmetric torsion, keep in mind that $\Theta$ and $\widetilde{\Theta}$ are both skew-symmetric tensors.
\end{proof}

\begin{remark}\label{remark_conexG2}
{\rm The torsion tensors of the invariant connections in (\ref{skew7}) are given by
\begin{equation*}
\begin{array}{rl}\vspace{1pt}
T^{\nabla}(X,Y)=&2 r \left(\Phi_{1}(X,Y)\,\xi_{1}-\eta_{1}(X)\psi_{1}(Y)+\eta_{1}(Y)\psi_{1}(X)\right)\\
&+
2\big(\mathrm{Re}(q)\Theta(X,Y)+\mathrm{Im}(q)\widetilde{\Theta}(X,Y)\big)
\end{array}
\end{equation*}
for every $X,Y \in \mathfrak{X}(\SS^{7})$. The related $3$-differential form $\omega_{_{\nabla}}$ can be written as follows
$$
\omega_{_{\nabla}}=\frac12 r\, \eta_{1}\wedge d\eta_{1}+ \mathrm{Re}(q) \left( \eta_{2}\wedge d\eta_{2}- \eta_{3}\wedge d\eta_{3}\right)- \mathrm{Im}(q) \left(\eta_{2}\wedge d\eta_{3}+\eta_{3}\wedge d\eta_{2}\right).
$$
For an arbitrary $7$-dimensional $3$-Sasakian manifold $M$, the $3$-differential form $$\frac{1}{2}(\eta_{1}\wedge d\eta_{1}+\eta_{2}\wedge d\eta_{2}+\eta_{3}\wedge d\eta_{3})+4\,\eta_{1}\wedge \eta_{2}\wedge \eta_{3}$$ is called the \emph{canonical} $G_{2}$-\emph{structure} of $M$  \cite{AgriFri}. There exists a unique metric connection $\nabla^{G_2}$   preserving the $G_{2}$-structure with totally skew-symmetric torsion 
$T^{G_2}$. This is called
the \emph{characteristic connection} of the $G_2$-structure.
The $3$-differential form $\omega_{_{G_2}}$ corresponding with the torsion $T^{{\tiny{G_2}}}$ is written in terms of the $3$-Sasakian structure by $\eta_{1}\wedge d\eta_{1}+\eta_{2}\wedge d\eta_{2}+\eta_{3}\wedge d\eta_{3}$. Therefore, the family of connections given in (\ref{skew7}) does not contain the characteristic connection of the canonical $G_2$-structure of $\SS^7$.  (In other words, neither $ T^{{\tiny{G_2}}}$ nor  $ \nabla^{{\tiny{G_2}}}$ are $\SU(4)$-invariant.)}
\end{remark}

\subsection{ $(\SS^7,g,\nabla)$ which are  $\nabla$-Einstein}

\begin{corollary}\label{777}
Let $\nabla$ be a $\SU(4)$-invariant metric affine connection with totally skew-symmetric torsion on $\SS^{7}$ given by {\rm (\ref{skew7})}. Then, the symmetric part of the Ricci tensor $\mathrm{Ric}^{\nabla}$ and the scalar curvature $s^{\nabla}$ are given by
$$
\begin{array}{l}\vspace{2pt}
\mathrm{Sym}(\mathrm{Ric}^{\nabla})=(6- 2 r ^2-4 |q|^{2})\, g +4 (|q|^{2}-  r ^2) \,\eta_{1} \otimes \eta_{1}, \\
   s^{\nabla}=6(7-3\, r ^2- 4\, |q|^{2}).
\end{array}
$$
In particular, $(\SS^{7},g,\nabla)$ is $\nabla$-Einstein if and only if $|q|^{2}=  r^2$.
In this case $\mathrm{Sym}(\mathrm{Ric}^{\nabla})=6(1-r^2)\, g$.
\end{corollary}

\begin{proof} Take $z\in \SS^{7} $ and $\mathcal{B}$ an orthonormal basis of $T_z\SS^{}$ as in (\ref{basebonita}). Analysis similar   to that in the proof of Corollary~\ref{333}, taking now into account the expression for the tensor $B$ provided in (\ref{tensorB}), shows for any $x,y\in T_z\SS^{7}$,
$$
\begin{array}{l}
S(x,y)=  (8r^2+16|q|^{2})g(x,y)+(16r^2-16|q|^{2})\eta_{1}(x)\eta_{1}(y).
\end{array}
$$
Hence we get $S=8(r^2+2|q|^{2})g+16(r^2-|q|^{2})\eta_{1}\otimes\eta_{1}$ and analogously $\|T^{\nabla}\|^2=12 r ^2+ 16|q|^{2}$. Now the formulas for $\mathrm{Sym}(\mathrm{Ric}^{\nabla})$ and $s^{\nabla}$ are direct consequences of (\ref{formulicas}), taking into consideration that  $\mathrm{Ric}^{g}=6g$ and $s^g=42$ for the sphere $\SS^7$. Finally, Equation~(\ref{einstein}) holds if and only $|q|^{2}=  r^2$.
\end{proof}

\begin{remark}\label{flatEinstein}
{\rm It is a classical result by \'{E}. Cartan and J.A. Schouten that a Riemannian manifold $M$ which admits a flat metric connection with totally skew-torsion splits and each irreducible factor is either a compact simple Lie group or the sphere $\SS^7$ \cite{CartanSchouten} (see \cite{AgriFri2} for a new proof and more references). A family of   connections on  $\SS^7$ satisfying such properties  is introduced  in \cite{AgriFri2}. In that paper, Agricola and Friedrich use an explicit parallelization of $\SS^7$ by orthonormal Killing vector fields $V_{1},...,V_{7}$ and then define the flat metric connection with skew-torsion $D$ by $DV_{i}=0$ for $i=1,...,7$. Of course, every Riemann-Cartan manifold $(M,g,\nabla)$ such that $\nabla$ is a flat metric connection with totally skew-torsion is  $\nabla$-Einstein.   

We would like to point out that we have obtained a family of  $\nabla$-Einstein manifolds   $(\SS^7,g,\nabla)$ (obtained for $|q|^{2}=   r^2$ in Corollary~\ref{777}), most of them with nonflat connections.
In case $r\ne\pm1$,  the Ricci tensor is immediately nonzero, in particular the related connections (for all values of $q$) are not flat.   The curvature tensors for the cases  $r=\pm1$, in which we have proved that $\mathrm{Sym}(\mathrm{Ric}^{\nabla})=0$, exhibit a different behaviour. More precisely, let us see that $\nabla$ is nonflat whenever $r\ne1$, and flat for $r=1$.

Let us consider the quaternionic Hopf principal bundle $\kappa\colon\SS^7 \to \mathbb{H}P^1$ with structural group $\SP(1)$. Recall that $\mathbb{H}P^1$ is diffeomorphic to $\SS^4$, $\SP(1)$ is isomorphic to $\SS^3$ and $\kappa$ is a Riemannian submersion with totally geodesic fibres and vertical distribution denoted by $\mathcal{V}$.
Let $X\in \mathfrak{X}(\SS^7)$ be a horizontal vector field such that $g(X,X)=1$ on an open subset $\mathcal{O}\subset \SS^7$. For every connection $\nabla$ given in (\ref{skew7}), its curvature $ R^{\nabla}$ on $\mathcal{O}$ satisfies
\begin{equation}\label{curvatura7}
R^{\nabla}(\xi_{2},\xi_{3})X= 2({ r-|q|^{2}}) \psi_{1}(X).
\end{equation}
In order to check (\ref{curvatura7}), recall  that $[X,\xi_{s}]\in \mathcal{V}$  since $X$ is a basic vector field and $\xi_s$ is vertical for $\kappa$ (see, for instance, \cite[9.23]{Besse}). Thus $\nabla^{g}_{\xi_{s}}X=-\psi_{s}(X)$ for $s=1,2,3$. Now a direct computation from Table 1 shows that
$$
\nabla_{\xi_{3}}X= ( {\mathrm{Re}(q)} -1 )\psi_{3}(X)+ {\mathrm{Im}(q)} \psi_{2}(X)
$$
and
$$
\nabla_{\xi_{2}}\nabla_{\xi_{3}}X=- ( {|q|^2} +1-2\mathrm{Re}(q) )\psi_{1}(X)+2\mathrm{Im}(q) X.
$$
In a similar way, we obtain the formulas
$$
\begin{array}{l}\vspace{1pt}
\nabla_{\xi_{3}}\nabla_{\xi_{2}}X=  ( {|q|^2} +1+2\mathrm{Re}(q) )\psi_{1}(X)+2\mathrm{Im}(q) X,\\
 \nabla_{[\xi_{2},\xi_{3}]}X=-2(1+r)\psi_{1}(X),
 \end{array}
$$
so that Equation~(\ref{curvatura7}) follows. In particular, for $r=-1$ and $q\in\SS^1$, we have got a family of nonflat connections satisfying $\mathrm{Sym}(\mathrm{Ric}^{\nabla})=0$ (moreover, Ricci-flat). 
For $r=1$, the formula~(\ref{curvatura7}) does not enable us to obtain any conclusion about the flatness.
So, it is the moment to take advantage of   Nomizu's Theorem again and use the
formula~(\ref{cur})   to derive the complete expression of the curvatures of the $\SU(4)$-invariant affine connections on $\SS^7$ corresponding to $r=\pm1$ (and providing $\nabla$-Einstein manifolds).  We work in  $T_o \SS^7\equiv\mathfrak{m}$ and extend by invariance,
by omitting some tedious computations.

In case $r=1$ (any $q\in\SS^1$), for any $z,w,u\in \mathbb{C}^3$ and $a,b,c\in{\bf i} \mathbb{R}$,
$$
R^{\nabla}((z,a),(w,b))(u,c)=(\bar z\times  (w\times u)-\bar w\times(z\times u)+z(\bar w^tu)-w(\bar z^tu)-u(\bar w^tz-\bar z^tw),0)=(0,0),
$$
so that we have a family of flat connections.

In case $r=-1$ (any $q\in\SS^1$),
for $X_1,X_2,X_3\in \mathfrak{X}(\SS^7)$, the curvature tensor is
$$ 
R^{\nabla}(X_{1}, X_{2})X_{3}= 4 \sum_{i=1}^3\mu(X_{i},X_{i+1},X_{i+2}) 
 -4\big(\mathrm{Re}(q) \Omega(X_1,X_2,\psi_1(X_3))+\mathrm{Im}(q) \Omega(X_1,X_2, X_3)\big)\xi_1 , 
$$
where the indices are taken modulo 3, and the tensor $\mu$ is given by
$$
\mu(X,Y,Z):= \psi_{1}(X)\Phi_{1}(Y,Z)-\eta_{1}(X)\big(  \mathrm{Re}(q)\widetilde{\Theta}(Y,Z)+ \mathrm{Im}(q) \Theta(Y,Z)\big).
$$   
It is interesting to remark that $R^{\nabla}(X,Y,Z)=R^{\nabla}(Y,Z,X)$, so that $R^\nabla$ is in fact a nonzero  totally skew-symmetric tensor.
This situation never happens to the   Riemannian tensor curvature of a Levi-Civita connection of a nonflat Riemannian manifold $M$,
as the first Bianchi identity shows.

Note that the 
formula~(\ref{cur}) also gives   that the Ricci tensor is symmetric independently of $r$ ($|q|^{2}=   r^2$). Hence
 $\mathrm{Ric}^{\nabla} =6(1-r^2)\, g$ is positive definite if and only if $r\in(-1,1)$ and negative definite when $|r|>1$. These facts are summarized in the next table.}
\end{remark}

\begin{table}[h]
\centering
  
\begin{tabular}{||c|| c|c|c|c|c|}
\hline  
 \vrule width 0pt height 12pt
$r$ & $(-\infty,-1)$ &$-1$ &$(-1,1)$ &$1$ &$(1,\infty)$\\[0.5ex]
\hline  
\vrule width 0pt height 12pt
 Ric$^\nabla$ & $<0$ & 0 & $>0$ & 0 &$<0$\\[0.5ex]
 \hline
 &$\ne0$&$\ne0$   & $\ne0$ & 0  &$\ne0$\\[-1ex]
 \raisebox{1.5ex}{$R^\nabla$} &  &totally skew   & $R^g$ for $r=0$  & flat  & \\
   \hline
   \end{tabular}\vspace{3pt}
  \caption{ }
  \label{tablaRiccis}
\end{table}

%%%%%%%%%%%%%%%%%%%%%%%%%%%%%%%%%%%%%%%%%%%%%%%%%%%%%%%%%%%%%%%%%%%%%%%%%%%%%%%%%%%%%%%%%%%%%%%%%%%%%
%%%%%%%%%%%%%%%%%%%%%%%%%%%%%%%%%%%%%%%%%%%%%%%%%%%%%%%%%%%%%%%%%%%%%%%%%%%%%%%%%%%%%%%%%%%%%%%%%%%%%%%%%%%%%%%%%%%%%%%%%%%%%%%%%%%%%%%%%%%%%%%%%%5
%%%%%%%%%%%%%%%%%%%%%%%%%%%%%%%%%%%%%%%%%%%%%%%%%%%%%%%%%%%%%%%%%%%%%%%%%%%%%%%

\section{Invariant connections on $\SS^{5}$  }

In this case we will also find the invariant tensors which will provide a bigger collection of invariant connections.
Despite of that,  there will not be nontrivial  invariant connections with skew-torsion satisfying the Einstein equation (\ref{einstein}).

\subsection{Invariant metric connections on $\SS^{5}$}

This time
the module $\mathfrak{m}_1^\mathbb{C}\cong V\oplus V^*$, where $\mathfrak{m}_1$ is the horizontal part of $\mathfrak{m}$,  decomposes as a sum of two isomorphic irreducible modules, because the natural module $V\cong \mathbb{C}^2$ and its dual one are isomorphic. Indeed, one can use a nonzero fixed bilinear alternating map  $\det\colon \Lambda^2V\to\mathbb{C}$ for the identification, or, alternatively, recall (for instance, from \cite[7.2]{Draper:Humphreysalg}) that there is just one irreducible $\mathfrak{sl}(2,\mathbb{C})$-module of each dimension $m+1$, usually denoted by $V_m$.

\begin{lemma}\label{CrisS5}
$
\mathrm{dim}_{\R} \hom_{\mathfrak{su}(2)}(\mathfrak{m}\otimes\mathfrak{m},\mathfrak{m})=13.
$
\end{lemma}

\begin{proof}
The decomposition of $\mathfrak{m}^\mathbb{C}$ as a direct sum of irreducibles modules is $2V_1\oplus V_0$, and in consequence 
$$
\mathfrak{m}^{\mathbb{C}}
\otimes\mathfrak{m}^{\mathbb{C}}\cong (4V_1\otimes V_1)\oplus (2V_1\otimes V_0)\oplus (2V_0\otimes V_1)\oplus (V_0\otimes V_0)
\cong 4V_2\oplus 4V_1\oplus 5V_0.
$$
Here there are 4 copies of $V\,(\cong V_1)$, 4 copies of $V^*$ (the same ones) and 5 copies of the trivial module, so that we can find $13$ linearly  independent homomorphisms from $\mathfrak{m}\otimes \mathfrak{m}$ to $ \mathfrak{m}$.
\end{proof}

In order to get an explicit basis of $\mathrm{Hom}_{\mathfrak{su}(2)}(\mathfrak{m}\otimes\mathfrak{m},\mathfrak{m})$, we only need an extra ingredient with respect to (\ref{eq_labase}). %A straightforward computation shows.

\begin{lemma}\label{invariantes5}
The map
$$\begin{array}{rccc}
\theta\colon &\mathbb{C}^2&\longrightarrow&\mathbb{C}^2\\
&z=\left(\begin{array}{l}z_1\\z_2\end{array}\right)&\mapsto &\theta(z)=\left(\begin{array}{l}-\bar z_2\\ \bar z_1\end{array}\right)
\end{array}
$$
is a homomorphism of $\mathfrak{su}(2)$-modules.
\end{lemma}

Therefore, the 13 independent bilinear maps next exhibited provide a basis of the vector space $\Gamma_2 $
of   $\mathbb{R}$-bilinear maps $\alpha\colon\mathfrak{m}\times\mathfrak{m}\to\mathfrak{m}$ such that $\mathrm{ad}(\mathfrak{su}(2))\subset \mathfrak{der}(\mathfrak{m},\alpha)$:
\begin{equation*} 
\begin{array}{ll}\vspace{1pt}
\alpha_{1}((z,a),(w,b))=(bz,0), & \alpha_{\bf{i}}((z,a),(w,b))=({\bf i}\,bz,0),
\\ \vspace{1pt}
  \beta_{1}((z,a),(w,b))=(aw,0), & \beta_{\bf{i}}((z,a),(w,b))=({\bf i}\,aw,0),
\\ \vspace{1pt}
 \gamma_{1}((z,a),(w,b))=(0,{\bf i}\,\mathrm{Im}(\overline{z}^t w)),
& \gamma_{\bf{i}}((z,a),(w,b))=(0,{\bf i}\,\mathrm{Re}(\overline{z}^t w)),
\\ \vspace{1pt}
\hat\alpha_{1}((z,a),(w,b))=(b\theta(z),0), &\hat\alpha_{\bf{i}}((z,a),(w,b))=({\bf i}\,b\theta(z),0),
\\ \vspace{1pt}
\hat\beta_{1}((z,a),(w,b))=(a\theta(w),0), &\hat\beta_{\bf{i}}((z,a),(w,b))=({\bf i}\,a\theta(w),0),
\\ \vspace{1pt}
\hat\gamma_{1}((z,a),(w,b))=(0,{\bf i}\,\mathrm{Im}(\overline{\theta(z)}^t w)),
&\hat\gamma_{\bf{i}}((z,a),(w,b))=(0,{\bf i}\,\mathrm{Re}(\overline{\theta(z)}^t w)),
\\ \vspace{1pt}
\delta((z,a),(w,b))=(0,{\bf i}ab).
\end{array}
\end{equation*}

The following step is to compute when a bilinear map $\alpha\in \Gamma_2$ is attached to an affine connection compatible with the metric. First, note that this family  (obtained following Remark~\ref{nnn}) has seven free parameters.

\begin{lemma}\label{le_compatiblesmetricaS5}
$
\mathrm{dim}_{\R} \mathrm{Hom}_{\mathfrak{su}(2)}(\mathfrak{m},\mathfrak{m}\wedge\mathfrak{m})=7.
$
\end{lemma}

\begin{proof}
After complexifying  $\mathfrak{m}^\mathbb{C}\cong 2V_1\oplus V_0$,
 it is not difficult to check that
$\mathfrak{m}^\mathbb{C}\wedge  \mathfrak{m}^\mathbb{C}\cong V_2\oplus2V_1\oplus 3V_0$.
Then, note that
$\dim \mathrm{Hom}_{\mathfrak{sl}(2,\C)}( 2V_1\oplus V_0, V_2\oplus2V_1\oplus 3V_0)=2\cdot2+1\cdot3=7$.

\end{proof}

 \begin{proposition}
A  $\SU(3)$-invariant affine connection $\nabla$ on $\SS^5$ is metric if and only if there are $q_{1},q_{2},q_{3}\in \C$ and $t\in \R$ such that the corresponding $\R$-bilinear map $\alpha_{_{\nabla}}\in \Gamma_{2}$ satisfies
\begin{equation}\label{mesa5}
\begin{array}{ll}
\alpha_{_\nabla}=&\mathrm{Re}(q_{1})(\alpha_{1}-\gamma_{1})+\mathrm{Im}(q_{1})(\alpha_{\bf{i}}+\gamma_{\bf{i}})+t\beta_{1}\\
&+\mathrm{Re}(q_{2})(\hat\alpha_1-\hat\gamma_1)+\mathrm{Im}(q_{2})(\hat\alpha_{\bf i}+\hat\gamma_{\bf i})+
\mathrm{Re}(q_{3})\hat\beta_1+\mathrm{Im}(q_{3})\hat\beta_{\bf i}.
\end{array}
\end{equation}
\end{proposition}
\begin{proof}
One can check that every map $\alpha_{_\nabla}$ as in Equation~(\ref{mesa5}) satisfies Equation~(\ref{metric}). The result follows from  Lemma \ref{le_compatiblesmetricaS5}.
\end{proof}

\begin{corollary}\label{torsion5} The torsion $T^{_{\nabla}}$ of the $\SU(3)$-invariant metric affine connection $\nabla$ on $\SS^5$ corresponding with the $\R$-bilinear map $\alpha_{_{\nabla}}\in \Gamma_2$ in {\rm (\ref{mesa5})} is characterized by
$$
\begin{array}{ll}\vspace{2pt}
T^{_{\nabla}}((z,a),(w,b))=&\left(\left(q_{1}-t-\frac{3}{2}\right)(bz-aw)+(q_{2}-q_{3})(b\,\theta(z)-a\,\theta(w)), 0 \right)
\\&+
\left(0, (\mathrm{Re}(q_{1})-1)(\overline{w}^tz-\overline{z}^tw)\right)-\left(0, 2{\bf i}\mathrm{Im}(q_{2}\overline{z}^{t}\theta(w))\right)
\end{array}
$$
for any $(z,a),(w,b)\in \mathfrak{m}$. In particular, the Levi-Civita connection  is achieved for $q_{1}=1$, $q_{2}=q_{3}=0$ and $t=-1/2$.
\end{corollary}
\begin{proof} Recalling that $[\,\,\, , \,\,\,]_{\mathfrak{m}}=\frac{3}{2}(\alpha_{1}-\beta_{1})-2\gamma_{1}$, then the proof is clear from (\ref{tor}).
\end{proof}

Again, in order to obtain the expressions for the metric invariant connections in this case, we have to consider a specific geometric structure in $\SS^5$.

Let us consider the Cayley numbers, or octonions,  $\mathbb{O}=\mathrm{Span}\{1,e_{1},...,e_{7}\}$, which is a real (nonassociative) division algebra with the product given by $e_1e_2=e_4$, also
$$
e_i^2=-1,\quad e_ie_j=-e_je_i\ (i\ne j),
 $$
 and, whenever $e_ie_j=e_k$, then $e_{i+1}e_{j+1}=e_{k+1}$ and $e_{2i}e_{2j}=e_{2k}$ (indices modulo 7).
 Fix the copy of the complex numbers   $\mathrm{Span}\{1,e_{7}\}$, so that $\mathbb{O}=\mathbb{C}\oplus \mathbb{C}\langle e_{1}, e_{2}, e_{4}\rangle$. 
 Thus, we identify $\C^3$ with $\mathbb{C}\langle e_{1}, e_{2}, e_{4}\rangle$ in a natural way mapping
\begin{equation}\label{idencomplexocto}
(z_{1},z_{2},z_{3}) \leftrightarrow z_{1}  e_{1}+z_{2}  e_{2}+z_{3}  e_{4}
\end{equation}
 where   $z_{l}\in \mathbb{R}\oplus \mathbb{R}e_{7}$ for $l=1,2,3$ and the juxtaposition denotes the product on $\mathbb{O}$.

Recall that an \emph{almost Hermitian structure} $(g,J)$ on a manifold $M$ is a Riemannian metric $g$ on $M$ and an almost complex structure $J$ such that $g(X,Y)=g(J(X),J(Y))$ for all $X,Y\in \mathfrak{X}(M)$. 
An almost Hermitian structure is said to be  \emph{nearly K\"{a}hler} whenever $(\nabla^{g}_{X}J)X=0$ for all $X\in \mathfrak{X}(M)$, where $\nabla^g$ denotes as usual the Levi-Civita connection of $g$. 
The standard example of nearly K\"{a}hler (non K\"{a}hler) manifold is the sphere $\SS^6$ as follows \cite[Chapter 4]{Blair}. Consider $\R^7$, in the natural way, as the imaginary part of the Cayley numbers $\mathbb{O}$, that is, $\mathrm{Im}(\mathbb{O})=\mathrm{Span}\{e_{1},...,e_{7}\}$. Consider the $2$-fold vector cross product on $\R^7$ given by $P(u,v)\equiv u\times v=\mathrm{Im}(u  v)=uv-\frac12\mathrm{tr}(uv)1$, for $\mathrm{tr}$ the trace given by $ \mathrm{tr}(u)=u+\bar u$.
Now consider the sphere $\SS^6$ in $\R^7$ with unit outward normal vector field $N$. Then an almost complex structure $J$ on $\SS^6$ is defined by $J(X)=P(N, X)=N\times X$ for every $X\in \mathfrak{X}(\SS^6)$. In this way $(\SS^6, g , J)$ is a nearly K\"{a}hler manifold.

Now we are in a position to describe the $\SU(3)$-invariant tensor which extends to the whole manifold $\SS^5$ the map $\theta$ given in Lemma~\ref{invariantes5} (details of the following construction can be consulted again in \cite[Chapter 4]{Blair}).
One starts from the totally geodesic embedding $\SS^{5}\to \SS^6$ defined by
$$
(z_{1},z_{2},z_{3})\in \SS^5\subset \C^3\mapsto (z_{1},z_{2},z_{3},0)\in \SS^6 \subset \mathrm{Im}(\mathbb{O})=\mathbb{C}\langle e_{1}, e_{2}, e_{4}\rangle\oplus\mathbb{R}e_7,
$$
which has unit normal vector field $\nu=-\frac{\partial}{\partial x_{7}}$. Thus, for every $z\in \SS^5$, the tangent space $T_{z}\SS^5$ can be seen as the hyperplane in $\mathrm{Im}(\mathbb{O})$ given by $x_{7}=0$. For $\SS^5$, the vector field $\xi$ already considered in (\ref{elcampo}) can be defined in an equivalent way as follows
$$
\xi_{z}=-e_{7}  (z_{1} e_{1}+z_{2}  e_{2}+z_{3}  e_{4})=-J(\nu_{z})
$$
for every $z=(z_{1},z_{2},z_{3})\in \SS^5$.

Recall that an  \emph{almost contact structure} on an odd dimensional manifold $M$ consists of a vector field $\xi'$, a $1$-form $\eta'$
and a field $\psi'$ of endomorphisms satisfying $\eta'(\xi')=1$ and $\psi'^2 =-\mathrm{Id}+\eta' \otimes \xi'$. When we have chosen a Riemannian metric $g$ such that $g(\psi'(X),\psi'(Y))=g(X,Y)-\eta'(X)\eta'(Y)$, the almost contact structure is said to be \emph{metric}.
From the above totally geodesic embedding of $\SS^5$ in $\SS^6$, we can induce an almost contact metric structure on $\SS^5$ different from the standard one underlying the Sasakian structure introduced in (\ref{sasaki}) (\cite[Example~4.5.3]{Blair}). For every $X\in \mathfrak{X}(\SS^5)$, the decomposition of $J(X)$ in tangent and normal components with respect to the above isometric embedding $\SS^5 \subset \SS^6$ determines the $(1,1)$-tensor field $\widehat{\psi}$  such that     
\begin{equation}\label{defTheta}
J(X)=\widehat{\psi}(X)+\eta(X)\nu,
\end{equation}
where   $\eta$ is the     $1$-differential form on $\SS^5$ already considered in  (\ref{sasaki}). 
Thus, $\xi$, $\eta$ and $\widehat{\psi}$ form another almost contact metric structure on $\SS^5$.  

\begin{lemma}
The endomorphism  $\widehat{\psi}_{o}\colon T_{o}\SS^5\to T_{o}\SS^5$ corresponds under the identification $\pi_{*}\colon\mathfrak{m}\to T_{o}\SS^5$ with the endomorphism of $\mathfrak{m}$ given by $(z,a)\mapsto(\theta(z),0)$.
\end{lemma}

\begin{proof}
Let us consider $x=(z,a)\in \mathfrak{m}\equiv T_{o}\SS^5\subset \C^3$ given by $z=(z_{1},z_{2})$, $a=e_7t$ with $z_{l}=x_{l}+e_7y_{l}$ and $x_l,y_l,t\in\mathbb{R}$ for $l=1,2$. 
 Note that $o=(0,0,1)\in \C^3$ corresponds with $e_{4}$ under the identification (\ref{idencomplexocto}). Thus, it is easy to check that
\begin{equation*}
\begin{array}{ll}
\widehat{\psi}(x)&=J(x)-\eta(x)\nu_{o}=e_{4}  (z_{1}  e_{1}+z_{2}  e_{2}+t\, e_{7}    e_{4})-(-t)(-e_{7})\\
&=
  e_{4}  (x_{1}\,e_{1}+y_{1}\,e_{3}+x_{2}\,e_{2}+y_{2}\,e_{6}+t\,e_{5})-t\,e_{7}
  \\
  &=
x_{1}\, e_{2}-y_{1}\,e_{6}-x_{2}\, e_{1}+y_{2}\,e_{3}=(-x_{2}+y_{2}  e_{7})  e_{1}+(x_{1}-y_{1}  e_{7})  e_{2}\\
  &=-\overline{z}_{2}  e_{1}+\overline{z}_{1}  e_{2}.
\end{array}
\end{equation*}
\end{proof}

Then we can write the $\mathfrak{su}(2)$-module homomorphisms as follows.

\begin{corollary}\label{cocina5}
The bilinear maps on $T_{o}\SS^{5}$ corresponding with $\hat{\alpha}_{1},\hat{\alpha}_{\bf{i}},\hat{\beta}_{1},\hat{\beta}_{\bf{i}}, \hat{\gamma}_{1}$ and $\hat{\gamma}_{\bf{i}}$ under the identification $\pi_{*}\colon \mathfrak{m}\to T_{o}\SS^{5}$  are given respectively by
$$
\begin{array}{ll}
\hat{\alpha}_{1}(x,y)=-\eta(y)\psi(\widehat{\psi}(x)),\quad \qquad &\hat{\alpha}_{\bf{i}}(x,y)=\eta(y)\Big(\widehat{\psi}(x)-\eta(\widehat{\psi}(x))\xi_{o}\Big)=\eta(y)\widehat{\psi}(x),
\\
\hat{\beta}_{1}(x,y)=-\eta(x)\psi(\widehat{\psi}(y)),\quad &\hat{\beta}_{\bf{i}}(x,y)=\eta(x)\Big(\widehat{\psi}(y)-\eta(\widehat{\psi}(y))\xi_{o}\Big)=\eta(x)\widehat{\psi}(y),
\\
\hat{\gamma}_{1}(x,y)=\Phi(\widehat{\psi}(x),y)\, \xi_{o},\quad &\hat{\gamma}_{\bf{i}}(x,y)=-g(\psi(\widehat{\psi}(x)),\psi(y))\,\xi_{o}=-g(\widehat{\psi}(x),y)\,\xi_{o},
\end{array}
$$
for all $x,y \in T_{o}\SS^{2n+1} $.
\end{corollary}

In order to ensure that $\widehat{\psi}$ is $\SU(3)$-invariant, let us recall that $\SS^6$ can be identified with a coset of the exceptional Lie group
$$
G_{2}=\mathrm{Aut}(\mathbb{O})=\{f:\mathbb{O}\to \mathbb{O}: \text{f is a $\R$-linear isomorphism and }f(x  y)=f(x)  f(y)\}
$$
in the following way. The natural action $G_{2}\times \SS^6\to \SS^6$ given by $(f,z)\mapsto f(z)$, where we think $z\in \SS^6\subset\mathrm{Im}(\mathbb{O})$, is transitive and the isotropy group $H$ of the element $e_{7}$ can be identified with $\SU(3)$. Indeed, every $f\in H$ is the identity map on the fixed copy of the complex numbers $\mathrm{Span}\{1,e_{7}\}\subset \mathbb{O}$. Thus, $f$ is completely determined by its values on $\mathbb{C}\langle e_{1}, e_{2}, e_{4}\rangle$ which can be endowed with a Hermitian product $\sigma$ by means of (\ref{idencomplexocto}), so that $\sigma(u,v)=g(u,v)-e_7g(e_7u,v)$. The isotropy group $H=\{f\in \mathrm{Aut}(\mathbb{O}): f(e_7)=e_7\}$ is isomorphic to
\begin{equation*}\label{g2}
 \{f\colon \mathbb{C}\langle e_{1}, e_{2}, e_{4}\rangle\to\mathbb{C}\langle e_{1}, e_{2}, e_{4}\rangle:f\text{ $\mathbb{C}$-linear and }\sigma (f(x),f(y))=\sigma (x,y)\} 
 \cong\SU(3)
\end{equation*}
by the assignment $f\mapsto f\vert_{\mathbb{C}\langle e_{1}, e_{2}, e_{4}\rangle}$.
Therefore, the sphere $\SS^6$ can be seen as the homogeneous space $G_{2}/\SU(3)$ in   such a way that the action of $G_{2}$ on $\SS^6$ preserves both the metric $g$ and the almost complex structure $J$ (see, for instance, \cite[4.1]{refS6} and  references therein).

Turning now to the sphere $\SS^5$ as a totally geodesic hypersurface of $  \SS^6$, the natural action of $G_2$ on $\SS^6$ restricts to an action of the isotropy group $H\cong\SU(3)$ on $\SS^5$. This action agrees with the usual action of $\SU(3)$ on $\SS^5$ described in Section~\ref{sec_esferasimpares}. In particular, the tensor $\widehat{\psi}$ introduced in (\ref{defTheta}) is $\SU(3)$-invariant.

\begin{theorem}\label{principalS5}
For every $\SU(3)$-invariant metric affine connection $\nabla$ on $\SS^{5}$, there exist $q_{1},q_{2},q_{3}\in\C$ and $t\in \R$ such that
$$
\begin{array}{ll}\vspace{2pt}
\nabla_{X}Y=&\nabla^{g}_{X}Y+(\mathrm{Re}(q_1)-1)(\Phi(X,Y)\,\xi+\eta(Y)\psi(X))+
\mathrm{Im}(q_1)(\nabla^{g}_{X}\psi)(Y)\\ \vspace{2pt}
&+
\mathrm{Re}(q_{2})\big(\eta(Y)\psi(\widehat{\psi}(X))+\Phi(\widehat{\psi}(X),Y)\xi\big)+\mathrm{Im}(q_{2})\big(-\eta(Y)\widehat{\psi}(X)+g(\widehat{\psi}(X),Y)\xi\big)
\\
&+\mathrm{Re}(q_{3})\,\eta(X)\psi(\widehat{\psi}(Y))-\mathrm{Im}(q_{3})\,\eta(X) \widehat{\psi}(Y)+(t+1/2)\,\eta(X)\psi(Y)
\end{array}
$$
for all $X,Y \in \mathfrak{X}(\SS^{2n+1})$.
Moreover, $\nabla$ has totally skew-symmetric torsion if and only if  there are $ r \in \R$ and $q \in \C$ such that
$$
\begin{array}{ll}\vspace{1pt}
\nabla_{X}Y=&\nabla^{g}_{X}Y+ r \left(\Phi(X,Y)\,\xi-\eta(X)\psi(Y)+\eta(Y)\psi(X)\right)\\\vspace{2pt}
&+
\mathrm{Re}(q)\,\big(\eta(Y)\psi(\widehat{\psi}(X))-\eta(X)\psi(\widehat{\psi}(Y))+\Phi(\widehat{\psi}(X),Y)\xi\big)\\ \vspace{2pt}
&+\mathrm{Im}(q)\,\big(\eta(X)\widehat{\psi}(Y)-\eta(Y)\widehat{\psi}(X)+g(\widehat{\psi}(X),Y)\xi\big).
\end{array}
$$
\end{theorem}
\begin{proof}
Since $\nabla^{g}$ is $\SU(3)$-invariant, the affine connection $\nabla$ is $\SU(3)$-invariant if and only if the difference tensor $\mathcal{D}=\nabla-\nabla^{g}$ so is. %is also $\SU(3)$-invariant.
According to the expression of $\alpha_{_{\nabla}}$ in (\ref{mesa5})  and by Lemma~\ref{cocina},  we get for every $x,y\in T_{o}\SS^{5}$,
$$
\begin{array}{rl}\vspace{1.5pt}
\mathcal{D}(x,y)=&L^{g}(x,y)-L^{\nabla}(x,y)=\alpha_{g}(x,y)-\alpha_{_{\nabla}}(x,y)
\\\vspace{2pt}
=&(\mathrm{Re}(q_{1})-1)(\Phi(x,y)\,\xi_{o}+\eta(y)\psi(x))+
\mathrm{Im}(q_{1})(\nabla^{g}_{x}\psi)(y)+(t+1/2)\eta(x)\psi(y)
\\\vspace{2pt}
& +\mathrm{Re}(q_{2})\big(\eta(y)\psi(\widehat{\psi}(x))+\Phi(\widehat{\psi}(x),y)\xi_{o}\big)+\mathrm{Im}(q_{2})\big(-\eta(y)\widehat{\psi}(x)+g(\widehat{\psi}(x),y)\xi_{o}\big)
\\\vspace{1pt}
& +\mathrm{Re}(q_{3})\eta(x)\psi(\widehat{\psi}(y))-\mathrm{Im}(q_{3})\eta(x) \widehat{\psi}(y).\end{array}
$$

Notice that the connection  $\nabla$ has totally skew-symmetric torsion if and only if the difference tensor $\mathcal{D}$ is skew-symmetric. This clearly forces $\mathrm{Im}(q_{1})=0$, $\mathrm{Re}(q_{1})=1/2-t$ and $q_{2}=-q_{3}$. The proof is completed by taking $ r =-t-1/2$ and $q=q_{2}$.
\end{proof}

As always, the torsion tensors of the metric affine connections with totally skew-symmetric torsion given in Theorem \ref{principalS5} are given by
$
T^{\nabla}(X,Y)=2\,(\nabla_{X}Y-\nabla^{g}_{X}Y)$.

\subsection{ $(\SS^5,g,\nabla)$ which are $\nabla$-Einstein    }

\begin{corollary}\label{555}
Let $\nabla$ be a $\SU(3)$-invariant metric affine connection with totally skew-symmetric torsion on $\SS^{5}$ described in Theorem \ref{principalS5}. Then, the symmetric part of the Ricci tensor $\mathrm{Ric}^{\nabla}$ is given by
$$
\begin{array}{l}
\mathrm{Sym}(\mathrm{Ric}^{\nabla})=  (4-2(r^2+|q|^2))\, g-2\, (r^2+|q|^2)\,\eta \otimes \eta.
\end{array}
$$
In particular,   $(\SS^5,g,\nabla)$ is not $\nabla$-Einstein for any $\SU(3)$-invariant metric affine connection unless $\nabla=\nabla^g$.
\end{corollary}

\begin{proof} Following the lines of the proof of Corollary \ref{333}, for any $z\in\SS^5$ and any $x,y\in T_z\SS^{5}$, we have
$$
S(x,y)=  8(r^2+|q|^{2})\big(g(x,y)+\eta(x)\eta(y)\big).
$$
The announced formulas for $\mathrm{Sym}(\mathrm{Ric}^{\nabla})$ are now deduced from (\ref{formulicas}). Then, it is not difficult to check that the tensor $\mathrm{Sym}(\mathrm{Ric}^{\nabla})$ is proportional to $g$ if and only if $r=q=0$, which corresponds with the Levi-Civita connection $\nabla^g$.
\end{proof}

\begin{remark}
{\rm   Bobie\'nski and Nurowski have also studied in \cite{casiS5perono} connections with skew-symmetric torsion in a $\SU(3)$-homogeneous space of dimension 5, namely, the irreducible symmetric space 
  $\SU(3)/\SO(3)$, called Wu space  (enclosed in their program on irreducible $\SO(3)$-geometry in dimension five).  
  Note that the Lie algebra of $\SO(3)$ is isomorphic to our fixed $\mathfrak{h}\cong\mathfrak{su}(2)$, but they are not conjugated as a subalgebras of $\mathfrak{su}(3)$ (our complement $\mathfrak{m}$ is not $\mathfrak{h}$-irreducible). That is, the same  $\mathfrak{g}$ and ``similar'' $\mathfrak{h}$ are related to completely different homogeneous spaces, even topologically (the sphere $\SS^5$ and the Wu space).
}
\end{remark}

%%%%%%%%%%%%%%%%%%%%%%%%%%%%%%%%%%%%%%%%%%%%%%%%%%%%%%%%%%%%%%%%%%%%%%%%%%%%%%%%%%%%%%%%%%%%%%%%%%%%%%%%%%%%%%%%%%%%%%%%%%%%%%%%%%%%%%%%%%%%%%%%%%5
%%%%%%%%%%%%%%%%%%%%%%%%%%%%%%%%%%%%%%%%%%%%%%%%%%%%%%%%%%%%%%%%%%%%%%%%%%%%%%%

\section{Invariant connections on $\SS^{3}$  }

The sphere $\SS^3$ is the only sphere (besides $\SS^1$) which can be endowed with a Lie group structure. Namely, it is diffeomorphic to the Lie group $\SU(2)$ through the map
$$
\SS^3 \to \SU(2), \quad (z,w)\in \SS^3  \mapsto \left( \begin{array}{cc}w&-\bar{z}\\ z& \bar{w}\end{array} \right)\in \SU(2).
$$
Thus, the point $o=(0,1)\in \SS^3$ corresponds to $\mathrm{I}_2\in \SU(2)$.
Observe that the tensor $g$ is a biinvariant metric tensor. As the isotropy group   is trivial,   the reductive decomposition (\ref{ttt})  is
$$
\mathfrak{g}=\mathfrak{m}=\mathfrak{su}(2)
=\left\{\left( \begin{array}{cc}-a&z\\-\bar z&a\end{array} \right): z\in\mathbb{C},a\in{\bf i}\mathbb{R}    \right\}\cong \mathbb{C}\oplus{\bf i}\mathbb{R}
$$
and $\mathfrak{h}=0$. Recall that, under the identification $\pi_{*}$, the Lie algebra $\mathfrak{su}(2)$ is identified with $T_{o}\SS^3$.
The Lie algebra $\mathfrak{su}(2)$ is the real span of the   traceless antihermitian matrices 
$$
E_{1}=\left( \begin{array}{cc} 0 & 1\\ -1 & 0 \end{array} \right), \quad E_{2}=\left( \begin{array}{cc} 0 & {\bf i}\\ {\bf i} & 0 \end{array} \right),\quad E_{3}=\left( \begin{array}{cc} -{\bf i} & 0\\ 0 & {\bf i} \end{array} \right),
$$
with Lie brackets given by $[E_{i},E_{i+1}]=-2E_{i+2}$ (indices modulo 3). We also denote by $E_{1},E_{2},E_{3}\in \mathfrak{X}(\SS^3)$ the corresponding left invariant vector fields, which provide a global orthonormal frame on $\SS^3$. Note that $E_{3}=-\xi$ and $\psi(E_{1})=E_{2}$.

Now, the required homomorphisms of $\mathfrak{h}$-modules in Theorem~\ref{nomizu} are simply homomorphisms of vector spaces. Thus, we have that $\dim_\mathfrak{h}\hom(\mathfrak{m}\otimes\mathfrak{m},\mathfrak{m})=(\dim\mathfrak{m})^3=27$ and each bilinear map 
$\alpha\colon\mathfrak{m}\times\mathfrak{m}\to \mathfrak{m}$ corresponds with an invariant affine connection on $\SS^3$.
The metric affine connections are in one-to-one correspondence with the $9$-dimensional vector space $\hom(\mathfrak{m},\mathfrak{so}(\mathfrak{m},g))$.
In order to provide a description of such connections,
consider the following $(1,1)$-tensor fields on $\SS^3$,
$$
\sigma^{3}=E_{1}^{\flat}\otimes E_{2}-E_{2}^{\flat}\otimes E_{1}, \quad \sigma^{1}=E_{2}^{\flat}\otimes E_{3}-E_{3}^{\flat}\otimes E_{2}, \quad \sigma^{2}=E_{3}^{\flat}\otimes E_{1}-E_{1}^{\flat}\otimes E_{3},
$$
which provide, when being evaluated at $o$, endomorphisms of $\mathfrak{m}$ under the identification $\pi^*$.
It is clear that $\mathfrak{so}(\mathfrak{m},g)=\mathrm{Span}\{\sigma^{1}_o  ,\sigma^{2}_o ,\sigma^{3}_o \}$ and $\mathrm{ad}_{E_{l}}=-2\sigma^{l}_o $ for $l=1,2,3$. Moreover, every $\sigma^{l}$ is $\SU(2)$-invariant since   $E_{l}$ is a left-invariant  vector field.
Also, let us consider $f_{ij}\colon \mathfrak{m}\to\mathfrak{so}(\mathfrak{m},g)$ the linear maps given by $f_{ij}(E_{s})=\delta_{is}\sigma^j_o $ for all $i,j,s\in \{1,2,3\}$ (where $\delta$ is the Kronecker  delta). The $(1,2)$-tensor fields on $\SS^3$ given by
 $F_{ij}(E_{s})=\delta_{is}\sigma^j$   are the key tools to describe the metric invariant connections on $\SS^3$. In fact, a $\mathbb{R}$-bilinear map $\alpha_{_{\nabla}}\colon\mathfrak{g}\times\mathfrak{g}\to\mathfrak{g}$ is attached to a metric affine connection $\nabla$ if and only if there are real numbers $t_{_{ij}}$ with $i,j\in \{1,2,3\}$ such that
 $$
 \alpha_{_{\nabla}}(x,y)=\sum_{l=1}^{3}E_{l}^{\flat}(x)\left(\sum_{i,j=1}^{3}t_{_{ij}}f_{_{ij}}(E_{l})(y)\right)=\sum_{i,j=1}^{3}t_{_{ij}}E_{i}^{\flat}(x)\sigma^{j}_o(y), 
 $$
 for all $x,y \in \mathfrak{g}$.
Since the Levi-Civita connection corresponds to $\alpha_g=\alpha_1-\delta_1-\beta_1$ as in Remark \ref{LV},
the difference tensor between $\nabla$ and   $\nabla^{g}$ satisfies
\begin{equation*}\label{eq_1deS3}
\mathcal{D}(X,Y)=-\eta(Y)\psi(X)-\Phi(X,Y)\xi+\eta(X)\psi(Y)-\sum_{i,j=1}^{3}t_{_{ij}}E_{i}^{\flat}(X)\sigma^{j}(Y)
\end{equation*}
for every $X,Y\in \mathfrak{X}(\SS^3)$.
In order to obtain the invariant connections $\nabla$ which have totally skew-symmetric torsion, a direct computation shows that
$$
\mathcal{D}(E_{1},E_{1})=t_{12}E_{3}-t_{13}E_{2}, \quad \mathcal{D}(E_{2},E_{2})=-t_{21}E_{3}+t_{23}E_{1}, \quad \mathcal{D}(E_{3},E_{3})=t_{31}E_{2}-t_{32}E_{1}.
$$
Therefore, under the assumption $\mathcal{D}$ is skew-symmetric, we get $t_{_{ij}}=0$ whenever $i\neq j$. Also,
$\mathcal{D}(E_{1},E_{2})=-(1+t_{11})E_{3}=-\mathcal{D}(E_{2},E_{1})=-(1+t_{22})E_{3}$, 
so that $t_{11}=t_{22}$ ($=t_{33}$) and

\begin{equation}\label{eq_lad}
\mathcal{D}(X,Y)=-\eta(Y)\psi(X)-\Phi(X,Y)\xi+\eta(X)\psi(Y)-t_{11}\, \sum_{i=1}^{3}E_{i}^{\flat}(X)\sigma^{i}(Y).
\end{equation}
These computations can be summarized as follows.

\begin{theorem}\label{principalS3}
A  $\SU(2)$-invariant affine connection $\nabla$ on $\SS^3$ is metric if and only if there are $t_{_{ij}}\in \R$ with $i,j\in \{1,2,3\}$ such that
\begin{equation}\label{eq_2deS3}
\nabla_{X}Y=\nabla^{g}_{X}Y+\sum_{i,j=1}^{3}t_{_{ij}}E_{i}^{\flat}(X)\sigma^{j}(Y)
\end{equation}
for all $X,Y \in \mathfrak{X}(\SS^{3})$.
Moreover, $\nabla$ has totally skew-symmetric torsion if and only if  there is $ r \in \R$ such that  
\begin{equation}\label{eq_3deS3}
\nabla_{X}Y=\nabla^{g}_{X}Y+r\left(\Phi(X,Y)\xi-\eta(X)\psi(Y)+\eta(Y)\psi(X)\right).
\end{equation}

Every metric invariant connection  on $\SS^3$ with totally skew-symmetric torsion satisfies the $\nabla$-Einstein equation, being
$$
\mathrm{Sym}(\mathrm{Ric}^{\nabla})=2(1-r^2)\, g.
$$
\end{theorem}

\begin{proof}
Take into account that  
$$
\sum_{i=1}^{3}E_{i}^{\flat}\otimes \sigma^{i}=(E_{2}^{\flat}\wedge E_{3}^{\flat})\otimes E_{1}+(E_{3}^{\flat}\wedge E_{1}^{\flat})\otimes E_{2}+(E_{1}^{\flat}\wedge E_{2}^{\flat})\otimes E_{3},
$$
but $((E_{2}^{\flat}\wedge E_{3}^{\flat})\otimes E_{1}+(E_{3}^{\flat}\wedge E_{1}^{\flat})\otimes E_{2})(X,Y)=-\eta(X)\psi(Y)+\eta(Y)\psi(X)$ and
$((E_{1}^{\flat}\wedge E_{2}^{\flat})\otimes E_{3})(X,Y)=\Phi(X,Y)\xi$. Thus the formulas (\ref{eq_2deS3})  and  (\ref{eq_3deS3})   follow from (\ref{eq_lad}).

In case $\nabla$ has totally skew-symmetric torsion, we can apply the proof of   Corollary~\ref{333}  to obtain the formula for $\mathrm{Sym}(\mathrm{Ric}^{\nabla})$, 
which in particular proves that  $(\SS^3,g,\nabla)$ is  $\nabla$-Einstein.
\end{proof}

\begin{remark}
{\rm Note that for a $3$-dimensional orientable Riemannian-Cartan manifold with skew-symmetric torsion, the $3$-form $\omega_{_\nabla}$ introduced in (\ref{tos}) must be a multiple of the volume form. This is our case for $\SS^3$. Indeed, it is a direct computation   that, for all $X,Y,Z\in \mathfrak{X}(\SS^3)$,
$$
g\big(\Phi(X,Y)\xi-\eta(X)\psi(Y)+\eta(Y)\psi(X),Z\big)=E^{\flat}_{1}\wedge E^{\flat}_{2} \wedge E^{\flat}_{3}(X,Y,Z),
$$
so that $\omega_{_\nabla}=2rE^{\flat}_{1}\wedge E^{\flat}_{2} \wedge E^{\flat}_{3}$.
}
\end{remark}

\begin{remark}
{\rm A complete study of the space of biinvariant affine connections on compact Lie groups can be found in \cite{otroLaquer}. There, Laquer shows that for a compact simple Lie group $G$, the space of biinvariant affine connections is one-dimensional in all cases except for $\SU(n)$ with $n\geq 3$. Therefore, for the special case of $\SU(2)\cong\SS^3$,   the space of  biinvariant affine connections is one-dimensional. Of course, this is not the case when we look only for left invariant connections, where we have a dependence on $27$-parameters.

For more results on invariant affine connections on Lie groups, we can mention that, viewing $G$   as a reductive homogeneous space of the group $G\times G$, there are three natural reductive decompositions which hence provide three different canonical connections (see \cite[p.~198]{Draper:KobayashiNomizu}).

}
\end{remark}

%%%%%%%%%%%%%%%%%%%%%%%%%%%%%%%%%%%%%%%%%%%%%%%%%%%%%%%%%%%%%%%%%%%%%%%%%%%%%%%%%%%%%%%%%%%%%%%%%%%%%%%%%%%%%%%%%%%%%%%%%%%%%%%%%%%%%%%%%%%%%%%%%%5
%%%%%%%%%%%%%%%%%%%%%%%%%%%%%%%%%%%%%%%%%%%%%%%%%%%%%%%%%%%%%%%%%%%%%%%%%%%%%%%
\section*{Appendix} 

Here  is a summary of the results of this work. %%55 millones de veces sin el "here there is"
\medskip

 \begin{center}
\begin{tabularx}{\textwidth}{| c  l | c | X |}
%\begin{tabularx}{15cm}{| c  l | c | X |}
\hline 
\multicolumn{4}{ |c| }{\ } \\ 
\multicolumn{4}{ |c| }{INVARIANT CONNECTIONS ON ODD DIMENSIONAL SPHERES} \\  
\multicolumn{4}{ |c| }{\ } \\ 
\hline \hline
 & Invariant &
7&
$\leftrightarrow\langle\{\alpha_{1},\,  \alpha_{\bf{i}},\,\beta_{1},\,\beta_{\bf{i}},\,\gamma_{1},\,  \gamma_{\bf{i}},\,\delta_1\}\rangle$
  \\  %\hdashline
  \cdashline{2-4} %\cr
  \vrule width 0pt height 14pt
  $\SS^{2n+1}$&Metric&3& {$\nabla^{g}_{X}Y+s_1(\Phi(X,Y)\,\xi+\eta(Y)\psi(X))  +
s_2(g(X,Y)\,\xi-\eta(Y)X)+s_3\eta(X)\psi(Y) $} 
\\
\cdashline{2-4} \vrule width 0pt height 14pt
  &Skew-Torsion&1& $ \nabla^{g}_{X}Y+s_1T^{c}(X,Y)$ \\ 
\cdashline{2-4} \vrule width 0pt height 14pt
&$\nabla$-Einstein&Point& $\nabla^{g}_{X}Y$
\\ \hline 
%
%\end{tabularx}
%
%\begin{tabularx}{\textwidth}{ |c|l|c|X|}
\hline
 & Invariant &
9&$\leftrightarrow\langle\{\alpha_{1},\,  \alpha_{\bf{i}},\,\beta_{1},\,\beta_{\bf{i}},\,\gamma_{1},\,  \gamma_{\bf{i}},\,\varepsilon_{1},\,\varepsilon_{\bf{i}},\,\delta_1\}\rangle$
  \\ 
 \cdashline{2-4}  \vrule width 0pt height 14pt
  $\SS^{7}$&Metric&5& $
\nabla^{g}_{X}Y+s_1\,  ( \Phi_1(X,Y)\,\xi_1+\eta_1(Y)\psi_1(X) )
 +s_2\,  \eta_1(X)\psi_1(Y)+
s_3\, \nabla^{g}_{X}\psi_1
+
s_4 \,  \Theta(X,Y)       +s_5 \,  \tilde\Theta(X,Y)
$
\\ 
\cdashline{2-4} \vrule width 0pt height 14pt
  &Skew-Torsion&3& $ \nabla^{g}_{X}Y+  s_1 T^{c}(X,Y)
 +
  s_4     \,   \Theta(X,Y)   +  s_5  \,  \tilde\Theta(X,Y)$  
 \\ 
 \cdashline{2-4}  \vrule width 0pt height 14pt
&$\nabla$-Einstein&Cone& $s_4^2+s_5^2=s_1^2$
\\   
\hline   
\hline
 & Invariant &
13&
$\leftrightarrow\langle\{\alpha_{1},\,  \alpha_{\bf{i}},\,\beta_{1},\,\beta_{\bf{i}},\,\gamma_{1},\,  \gamma_{\bf{i}},\,\hat\alpha_{1},\,  \hat\alpha_{\bf{i}},\,\hat\beta_{1},\,\hat\beta_{\bf{i}},\,\hat\gamma_{1},\,  \hat\gamma_{\bf{i}},\,\delta_1\}\rangle$
  \\ 
\cdashline{2-4}  \vrule width 0pt height 14pt
  $\SS^{5}$&Metric&7 & $\nabla^{g}_{X}Y+s_1(\Phi(X,Y)\,\xi+\eta(Y)\psi(X))  +
s_2\eta(X)\psi(Y)+ s_3(\Phi(\widehat{\psi}(X),Y)\,\xi+\eta(Y)\psi(\widehat{\psi}(X))) +s_4\eta(X)\psi(\widehat{\psi}(Y)) +s_5\eta(X)\widehat{\psi}(Y)+
s_6(g(\widehat{\psi}(X),Y)\,\xi-\eta(Y)\widehat{\psi}(X))+s_7 \nabla^{g}_{X}\psi$ 
\\ 
\cdashline{2-4} \vrule width 0pt height 14pt
  &Skew-Torsion&3& $ \nabla^{g}_{X}Y
  +s_1\big(\Phi(\widehat{\psi}(X),Y)\,\xi-\eta(X)\psi(\widehat{\psi}(Y))+\eta(Y)\psi(\widehat{\psi}(X))\big)+s_2 T^{c}(X,Y)
  +s_3\big( g(\widehat{\psi}(X),Y)\,\xi+\eta(X)\widehat{\psi}(Y)-\eta(Y)\widehat{\psi}(X)\big)$  
  \\ 
\cdashline{2-4}  \vrule width 0pt height 14pt
&$\nabla$-Einstein&Point& $\nabla^{g}_{X}Y$ 
\\ 
\hline
\hline
 & Invariant &
27& $   \nabla^{g}_{X}Y+\sum s_{ijk}E_i^{\flat}(X)E_j^{\flat}(Y)E_k   $
 \\ 
 \cdashline{2-4}  \vrule width 0pt height 14pt
  $\SS^{3}$&Metric&9& $ \nabla^{g}_{X}Y+\sum s_{ij}E_i^{\flat}(X)(E_{j}^{\flat}(Y)E_{j+1}-E_{j+1}^{\flat}(Y)E_{j} )$  
  \\ 
  \cdashline{2-4} \vrule width 0pt height 14pt
  &Skew-Torsion&1& $ \nabla^{g}_{X}Y+s_1T^{c}(X,Y)$  
  \\ 
  \cdashline{2-4} \vrule width 0pt height 14pt
&$\nabla$-Einstein&Line&  $ \nabla^{g}_{X}Y+s_1 T^{c}(X,Y)
$  \cr\hline 
\end{tabularx}
 \end{center}

\vspace{3pt}
\noindent
where $T^{c}$ denotes the torsion tensor of the characteristic connection of the Sasakian manifold given in Example \ref{222}.iii).

\section*{Acknowledgments} 
The authors are greatly indebted to Alberto Elduque for his nice and comprehensive notes about Nomizu's Theorem (essentially compiled in \cite{ApuntesAlbertoNomizu}), %which have extensively been used in Section~2,  
and for some  valuable hints on Lemma~\ref{algebra}.

\end{document}